\documentclass[ssy]{imsart}

\usepackage{amsthm,amsmath,amssymb,bbm,natbib}
\usepackage{enumitem}
\arxiv{arXiv:1312.2340}

\startlocaldefs

\newcommand\testStep[1]{%
\ifcase#1 % 0
I don't know what to do with this!
\or% 1
First step
\or% 2
Second step
\or% 1
Third step
\or% 2
Fourth step
\or% 1
Fifth step
\or% 2
Sixth step
\or% 1
Seventh step
\or% 2
Eighth step
\or% 1
Ninth step
\or% 2
Tenth step
\else% >2
I don't know what to do with this!
\fi}

\newcounter{step}
\newcommand{\mysubsection}[1]{\subsection{\testStep{\thestep}: #1}\addtocounter{step}{1}}

\def\d{{\normalfont{\textrm{d}}}}
\def\zero{{\bf{z}}}
\def\RP{{\textnormal{(R)}}}
\def\Var{{\textnormal{Var}}}

\def\E{{\mathbb{E}}}
\def\Ebf{{\mathbf{E}}}
\def\N{{\mathbb{N}}}
\def\P{{\mathbb{P}}}
\def\Pbf{{\mathbf{P}}}

\def\R{{\mathbb{R}}}
\def\T{{\mathbb{T}}}

\def\stree{{\normalfont{\texttt{t}}}}
\def\tree{{\normalfont{\texttt{T}}}}
\def\Tree{{\mathcal{T}}}

\def\Ccal{{\mathcal{C}}}
\def\Dcal{{\mathcal{D}}}
\def\Ecal{{\mathcal{E}}}
\def\Fcal{{\mathcal{F}}}
\def\Gcal{{\mathcal{G}}}
\def\Hcal{{\mathcal{H}}}

\def\Kcal{{\mathcal{K}}}
\def\Lcal{{\mathcal{L}}}
\def\Mcal{{\mathcal{M}}}
\def\Ncal{{\mathcal{N}}}

\def\Scal{{\mathcal{S}}}
\def\Tcal{{\mathcal{T}}}

\newcommand{\indicator}[1]{{\mathbbm{1}}_{\{#1\}}}

\newtheorem{theorem}{Theorem}[section]

\newtheorem{lemma}[theorem]{Lemma}
\newtheorem{corollary}[theorem]{Corollary}

\newtheorem{rk}[theorem]{Remark}

\endlocaldefs

\begin{document}

\begin{frontmatter}

% "Title of the paper"
\title{Scaling limit of a limit order book model via the regenerative characterization of L\'evy trees}
\runtitle{Scaling limit of a limit order book model}

% indicate corresponding author with \corref{}
% \author{\fnms{John} \snm{Smith}\corref{}\ead[label=e1]{smith@foo.com}\thanksref{t1}}
% \thankstext{t1}{Thanks to somebody} 
% \address{line 1\\ line 2\\ printead{e1}}
% \affiliation{Some University}

\author{\fnms{Peter} \snm{Lakner}\ead[label=e1]{plakner@stern.nyu.edu}}
\address{Leonard N. Stern School of Business\\New York University\\44 W 4th St\\New York, NY, 10012\\USA}
% \affiliation{???}
\and
\author{\fnms{Josh} \snm{Reed}\ead[label=e2]{jreed@stern.nyu.edu}}
\address{Leonard N. Stern School of Business\\New York University\\44 W 4th St\\New York, NY, 10012\\USA}
\and
\author{\fnms{Florian} \snm{Simatos}\ead[label=e3]{florian.simatos@isae-supaero.fr}}
\address{ISAE-SUPAERO\\10 avenue Edouard Belin\\31055 Toulouse\\France}
% \affiliation{???}

\runauthor{P.\ Lakner, J.\ Reed and F.\ Simatos}

\begin{abstract}
	We consider the following Markovian dynamic on point processes: at constant rate and with equal probability, either the rightmost atom of the current configuration is removed, or a new atom is added at a random distance from the rightmost atom. Interpreting atoms as limit buy orders, this process was introduced by~\cite{Lakner16:0} to model a one-sided limit order book.
	
	We consider this model in the regime where the total number of orders converges to a reflected Brownian motion, and complement the results of~\cite{Lakner16:0} by showing that, in the case where the mean displacement at which a new order is added is positive, the measure-valued process describing the whole limit order book converges to a simple functional of this reflected Brownian motion. Our results make it possible to derive useful and explicit approximations on various quantities of interest such as the depth or the total value of the book.
	
	Our approach leverages an unexpected connection with L\'evy trees. More precisely, the cornerstone of our approach is the regenerative characterization of L\'evy trees due to~\cite{Weill07:0}, which provides an elegant proof strategy which we unfold.
\end{abstract}

%\begin{keyword}[class=MSC]
%\kwd[Primary ]{}
%\kwd{}
%\kwd[; secondary ]{}
%\end{keyword}

%\begin{keyword}
%\kwd{}
%\kwd{}
%\end{keyword}

\end{frontmatter}

\section{Introduction}

\subsection{Context}

The limit order book is a financial trading mechanism that facilitates the buying and selling of securities by market participants. It keeps track of orders made by traders, which makes it possible to fulfill them in the future. For instance, a trader may place an order to buy a security at a certain level $p$. If the price of the security $\pi$ is larger than $p$ when the order is placed, then the order is kept in the book and will be fulfilled if the price of the security falls below $p$.

Due to its growing importance in modern electronic financial markets, the limit order book has attracted a significant amount of attention in the applied probability literature recently. One may consult, for instance, the survey paper by~\cite{Gould13:0} for a list of references. Several mathematical models of the limit order book have been proposed in recent years, ranging from stylized models such as the Stigler-Luckock model (see~\cite{Kelly:0, Luckock03:0, Swart:0}) to more complex models such as those proposed by~\cite{Cont10:0} or~\cite{Gareche13:0}. Broadly speaking, these models may be categorized as being either discrete and closely adhering to the inherent quantized nature of the limit order book, or as being continuous in order to better capture the high frequency regime in which the order book typically evolves.

In the present paper, we attempt to bridge the gap between the discrete and continuous points of view by establishing the weak convergence of a discrete limit order book model to a continuous one in an appropriately defined high frequency regime where the speed at which orders arrive grows large. Similar weak convergence results have recently been considered in various works. However, most of the time, only finite-dimensional statistics of the limit order book are tracked such as the bid and ask prices (the highest prices associated with a buy and sell order on the book) or the spread (equal to the difference between these two quantities), see for instance~\cite{Abergel13:0, Blanchet:0, Cont:0}; \cite{Cont13:0, Kirilenko13:0}. In contrast, in the present paper we establish the convergence of the full limit order book which we model by a measure-valued process. This approach has also been taken in~\cite{Osterrieder07:0}. In~\cite{Horst17:0} the authors also model the entire book but with a different approach, namely, they track the density of orders which they see as random elements of an appropriate Banach space.

\subsection{Multiplicative model description and main result} \label{sub:model-description}

The discrete model that we study is a variant of the limit order book model proposed by~\cite{Lakner16:0}. This is a one-sided limit order book with only limit buy orders which are therefore fulfilled by market sell orders. Since in this one-sided case no confusion can arise, in the rest of the paper we simply refer to them as limit orders and market orders, and we call \emph{price} the bid price, i.e., the highest price associated with a limit order in the book. Limit and market orders arrive according to two independent Poisson processes, and:
\begin{description}
	\item [Upon arrival of a market order] the market order fulfills one of the limit orders associated to the current price which is therefore removed from the book;
	\item [Upon arrival of a limit order] conditionally on the state of the book, the location of this new order is distributed like $\max(M p, p_0)$ with $p$ the current price, $p_0 > 0$ some fixed parameter and $M$ some positive random variable.
\end{description}
The parameter $p_0$ models the action of a market maker that prevents the price from reaching arbitrary low values. Our techniques could be extended to allow for $p_0 = 0$: in this case, $p_0$ has to be replaced by the past infimum of the price, see the discussion in Section~\ref{sec:discussion}.
% If $p_0 = 0$, then $X$ appropriately scaled should converge to the process
% \[ \left( \frac{1}{\E(\log M)} \Lambda'(W'_t, \underline W'_t), t \geq 0 \right) \]
% with $W'$ a Brownian motion (without reflection), $\underline W'_t = \inf_{[0,t]} W'$ the past infimum process and $\Lambda(w, \underline w)$ for each $0 \leq \underline w \leq w$ the measure which acts on bounded measurable functions $f: \R \to \R$ as follows:
% \[ \int f(x) \Lambda'(w, \underline w)(\d x) = \int_{e^{\underline w}}^{e^w} \frac{f(x)}{x} \d x. \]

In our model, a new limit order is only placed in the vicinity of the current price: this assumption is justified from empirical evidence that shows that the major component of the order flow occurs at the (best) bid and ask price levels, see for instance~\cite{Biais95:0}. This feature has been incorporated in previous models such as~\cite{Cont:0, Cont14:0} and the purpose of our model is to understand the impact of this fundamental behavior in the high-frequency regime.
\\

We model the state of the book at time $t$ by a finite point measure $X_t$ on $\R_+$ where atoms record positions of orders in the book. Our main result (Theorem~\ref{thm:main} below) states that when $\E(\log M) > 0$ and in the high-frequency regime where the rates at which the market and limit orders arrive grow large, then the process $(X_t, t \geq 0)$ appropriately scaled converges to the measure-valued process
\[ \left( \frac{1}{\E(\log M)} \Lambda(W_t), t \geq 0 \right) \]
with $W$ a Brownian motion reflected at $0$ and $\Lambda(w)$ for each $w \geq 0$ the measure which acts on bounded measurable functions $f: \R \to \R$ as follows:
\begin{equation} \label{eq:limit-mul}
	\int f(x) \Lambda(w)(\d x) = \int_{p_0}^{p_0 e^w} \frac{f(x)}{x} \d x.
\end{equation}
Note that $W$ has variance $(2 \lambda) [\E(\log M)]^2$ and is allowed to have a drift $m \in \R$: $\lambda$ is the asymptotic rate at which orders arrive and $m$ emerges as the difference between the arrival rates of market and limit orders. The drift thus reflects the imbalance between offer and demand which is line with the standard economic models of price evolution such as in~\cite{Foucault05:0}.

In the original discrete model, the price at time $t$ corresponds to the supremum of the support of $X_t$. In the limit, we define the limiting price process $(\pi_t, t \geq 0)$ in the same way from $((1/\E(\log M)) \Lambda(W_t), t \geq 0)$, i.e., $\pi_t = \sup {\rm supp} (\Lambda(W_t))$.

\begin{rk}
	Here we have stated a multiplicative version of our result where the location of an order is obtained as a multiplicative factor of the current price. This is the most relevant form of our result from an application point of view and is thus suited for the discussion below. However, from a technical standpoint it is more convenient to take a logarithmic transformation and consider an additive model where the displacement is obtained as a linear addition to the current pice. This will be the setting adopted in Section~\ref{sec:model:main} onward, and in particular in the statement of our main result, Theorem~\ref{thm:main} below.
\end{rk}

\subsection{Insights into limit order book}

This result brings practical insight into the behavior of limit order books in the high frequency regime as we now discuss.

\subsubsection{State-space collapse, local evolution and asymptotic Markovianity of the price} \label{subsub:SSCandco}

Our limiting process $((1/\E(\log M)) \Lambda(W_t), t \geq 0)$ has the striking feature that the associated price process $(\pi_t, t \geq 0)$ is Markov, namely it is a geometric reflected Brownian motion. Note that this is in accordance with standard assumptions in finance such as in the famous model of Black and Scholes. What is more, it actually bears all the randomness since according to~\eqref{eq:limit-mul}, $\Lambda(W_t)$ is a deterministic function of~$\pi_t$. This phenomenon of dimension reduction, going from a process with values in the space of measures to a real-valued process, is well-known in queueing theory where it is referred to as the state-space collapse phenomenon~\cite{Bramson98:0, Gromoll04:0, Reiman84:0, Williams98:0}. To the best of our knowledge it is however the first time that it is observed in financial applications.

These two properties are surprising because they are far from being true in the original discrete model. Indeed, in the original discrete model it is not enough to know the price process to know the entire state of the book, and the price process is not amenable to a simple description: it is not even Markov. Actually, most of the forthcoming technical difficulties come from the fact that in order to control the discrete price process, one needs to know the entire state of the book.

A closely related striking feature is that in the limit, the price process becomes symmetric in the following sense. At the discrete level, the price process increases when a limit order arrives and draws a random variable $M > 1$: in this case, the (multiplicative) increase is independent of the state of the book and is distributed like $M$ conditioned on being $>1$. On the other hand, the price can only decrease when a market order arrives and in this case the decrease of the price is governed by the state of the book. Intuitively, the more orders in the book the smaller the decrease since it will be more likely for an order to be close to the current price. However, this asymmetry is washed out in the limit: the fact that the limiting price process is a geometric reflected Brownian motion implies that it behaves as if the increase and decrease were distributed identically. Note that, heuristically at least, this gain in symmetry is necessary for the price process to become Markovian since otherwise, the evolution of the price process would depend on the state of the book. The local evolution of the price and its asymptotic Markovianity are therefore closely related.

As mentioned in Section~\ref{sub:model-description}, our model is meant to shed light on the impact of the fact that orders are placed in the vicinity of the current price: in light of the above discussion, we believe that one of the insights of our result is to justify the use of Markovian models for price evolution in the high-frequency regime, even though this would not be a reasonable assumption at the discrete level. We note moreover that our model leads to the price process following a geometric reflected Brownian motion which is in accordance with standard models~\cite{Hull18:0}.

\subsubsection{Convergence of the entire book and useful approximations} \label{subsub:conv-functionals}

Furthermore, establishing convergence of the measure-valued process describing the whole state of the book bears at once all the relevant information on the limit order book. For instance, our result combined with the continuous mapping theorem implies the convergence of:
\begin{description}
	\item [The price process:] the price process converges toward the stochastic process $(\pi_t, t \geq 0)$ with $\pi_t$ the supremum of the support of the measure $\Lambda(W_t)$ and so, as mentioned above, $\pi$ is a geometric reflected Brownian motion. In particular, we know its law which is for instance given in the case of zero drift by $e^{\lvert N \rvert}$ with $N$ a normal random variable, which allows for explicit computation of its mean, variance, etc;
	\item [The depth of the book:] the depth of the book converges to the process $(\pi_t - p_0, t \geq 0)$;
	\item [The total value of the book:] in the discrete model, it is defined as the sum of the prices corresponding to the limit orders in the book. In terms of the measure $X_t$, this corresponds to the mass $\int x X_t(\d x)$ which thus converges under the appropriate scaling toward
	\[ \int x \Lambda(W_t)(\d x) = \frac{1}{\E(\log M)}(\pi_t - p_0). \]
\end{description}
Further examples of functionals that converge include the time-to-fill of an order in the book or the fill probability. From a practical standpoint, these convergence results yield approximations of discrete quantities by continuous ones in the high-frequency regime. These approximations are particularly useful because, as discussed above, the limiting continuous process is much simpler than the original discrete one since it is a deterministic function of a geometric Brownian motion.

For simplicity we have restricted the proof of our main result to the case where $W$ has no drift. With minor technical additions they could be extended to allow for an arbitrary drift $m \in \R$ which is the setting considered for the following discussion, see Section~\ref{sec:discussion} for more details. We then have the following high frequency approximations where we write $\sigma^2 = (2 \lambda) [\E(\log M)]^2$ for the variance of $W$:
\begin{description}
	\item[Time-to-fill density:] let $T$ be the time-to-fill of an order at $p$ when the current price is at $\pi > p$: then our result yields the approximation
	\[ \P(T \in \d t) \approx \frac{\log(\pi/p)}{\sigma \sqrt{2 \pi t^3}} \exp \left( - \frac{(\log(\pi/p) + mt)^2}{2 \sigma^2 t} \right); \]
	\item[Fill probability:] the previous approximation directly leads to an approximation for the fill probability of a similar order, namely
	\[ \P(T < \infty) \approx \begin{cases}
		1 & \text{ if } \ m \leq 0,\\
		\exp \left( - \frac{2 m \log(\pi / p)}{\sigma} \right) & \text{ else}.
	\end{cases} \]
\end{description}
These high frequency approximations are obtained by solving the similar problem on our limiting process $((1/\E(\log M)) \Lambda(W_t), t \geq 0)$. This Brownian approximation of the price process also tells us that in the high-frequency regime, the probability for the price to go up or down is actually independent of the state of the book and is close to $\frac{1}{2}$.

% \subsubsection{Control problem}

\subsection{Link with other strands of applied probability} Our model is motivated by financial applications but it can also be viewed as a queueing model and as a particular model of branching random walk: we now detail the potential contribution of our paper for these two fields. Note also that, as will be discussed in Section~\ref{subsub:LIFO} below, the techniques we develop give insight into the heavy-traffic approximation of the LIFO queue established by~\cite{Limic00:0, Limic01:0}.

\subsubsection{Spatial queueing models} There has recently been a surge of interest for spatial queueing models where either the server moves and/or users are spread in space. The references~\cite{Altman94:0, Coffman87:0, Foss15:0} belong to the first category and investigate the performance of various spatial service disciplines as the server moves and customers arrive in space, while the references~\cite{Aldous17:0, Bouman11:1, Ven13:0} belong to the second category and analyze, among others, the performance of various classical medium access mechanisms when customers are spread in space.

In terms of spatial queueing models, our model can be formulated as follows: there is a queue of customers lined up on $\R$, with only the rightmost customer being served. When a new customer enters, she chooses her position at a random distance from the rightmost customer, i.e., the customer currently in service. If the new customer becomes the rightmost one, then she receives service, i.e., the service discipline is pre-emptive. If customers were always picking a position to the right of the rightmost customer, then this queueing system would essentially be a LIFO queue. However, because customers may choose a position to the left of the customer currently in service this creates an intricate correlation between service and spatial positions which we analyze here.

\subsubsection{Link with branching random walks}

In a branching random walk, individuals reproduce as in a Galton--Watson branching process and also undergo spatial motions with the constraint that children start off at the same location as their parent. This is an important model in applied probability with connections to, e.g., partial differential equations and Gaussian fields, see for instance~\cite{Bramson78:1} and~\cite{Ding14:0}.

Many variants of this model have been studied, such as branching random walks with a barrier (introduced by~\cite{Biggins91:0}) or more recently the $N$ branching random walk where only the $N$ rightmost particles are allowed to reproduce and the others are killed as in~\cite{Brunet97:0}. As will be recalled below, in~\cite{Simatos14:1} a coupling was established between the model of the present paper and branching random walks, which makes it possible to see our model as a branching random walk where only the rightmost particle is allowed to reproduce, and does so one child at a time.

\subsection{Relation with previous work}

As mentioned earlier, we consider in the present paper the case $\E(\log M) > 0$. The case $\E(\log M) < 0$ is studied in~\cite{Lakner16:0} and exhibits a fundamentally different behavior. Namely, when $\E(\log M) < 0$ the entire limit order book appropriately scaled is asymptotically concentrated at the price and the latter converges to a monotonically decreasing process. The sharp contrast between the cases $\E(\log M) > 0$ and $\E(\log M) < 0$ reflects the intrinsic asymmetric nature of the discrete limit order book model itself, discussed in Section~\ref{subsub:SSCandco}.

% In the present paper we complement this result by considering the case $\E(\log M) > 0$. In stark contrast to the case $\E(\log M) < 0$, our main result (Theorem~\ref{thm:main} below) shows that when $\E(\log M) > 0$, the price process converges to a reflected Brownian motion and that at any point in time the measure describing the book puts mass on a non-empty interval.
% 
% It is worthwhile to note that in the high frequency regime that we consider, the total number of orders in the book converges to a reflected Brownian motion and that in both cases $\E(\log M) < 0$ and $\E(\log M) > 0$, the limiting measure is a deterministic function of this reflected Brownian motion. However, this deterministic function changes completely depending on the sign of $\E(\log M)$, and this interesting dichotomy reflects the asymmetric nature of the discrete limit order book model itself. Another interpretation of this dichotomy is discussed at the end of Section~\ref{sec:discussion}.
% 
The case $\E(\log M) < 0$ is treated in~\cite{Lakner16:0} using stochastic calculus arguments, and the main challenge of the proof is to show that the price process is asymptotically monotone. When $\E(\log M) > 0$ this is no longer the case, and the main challenge of the present paper is to show that the price process converges to a geometric reflected Brownian motion. Because of this fundamentally different asymptotic behavior, the techniques developed in~\cite{Lakner16:0} cannot be directly applied to study the case $\E(\log M) > 0$.

For this reason, we develop here entirely new arguments relying on the regenerative property of~\cite{Weill07:0}. To unfold this proof strategy, we need to control specific random times, typically left and right endpoints of excursions. These controls are provided by a coupling laid down in~\cite{Simatos14:1} between the model studied here and a branching random walk. Conceptually, this coupling is merely a technical tool that is only used to control these random times: it is conceivable that they could be controlled by other means and the coupling avoided altogether. In contrast, the connection with L\'evy trees which we present briefly now, and in more details in Section~\ref{sub:link-levy-trees}, lies at the heart of our approach and constitutes an original contribution which could be useful for other stochastic systems.

\subsection{An unexpected connection with L\'evy trees and the heavy traffic limit of the LIFO queue revisited}

\subsubsection{A new characterization of the reflected Brownian motion}

Quite surprisingly, we find that in our model the log-price process satisfies a regenerative property very close to the one characterizing the contour process of Galton--Watson trees, see Lemma~\ref{lemma:discrete-reg} and the discussion following it. Roughly speaking, this regenerative property says that for any level $a > 0$, successive excursions above level $a$ are i.i.d.\ and that their common law does not depend on $a$.

Clearly this property is satisfied by a reflected Brownian motion and one of the takeaway from~\cite{Weill07:0} is that, under certain conditions, this is the only process that satisfies this property. In other words, this regenerative property can be used to characterize a reflected Brownian motion.

Since the discrete log-price process almost satisfies this regenerative property, in order to prove that it converges to a reflected Brownian this suggests to prove that any of its accumulation points satisfies this regenerative property. Although this argument lays down an elegant and attractive proof strategy, many technical details need to be taken care of along the way and Section~\ref{sec:proof} of the paper is dedicated to working these details out.

\subsubsection{The heavy traffic limit of the LIFO queue revisited} \label{subsub:LIFO}

To conclude the presentation and implications of our results, we note that the approach explained above, using the regenerative characterization of the reflected Brownian established by~\cite{Weill07:0}, provides a new interpretation for the results by~\cite{Limic00:0, Limic01:0} where it is proved that the scaling limit of the pre-emptive LIFO queue is the height process of a L\'evy tree.

First of all, note that the queue length process of the pre-emptive LIFO queue indeed exhibits the above regenerative property (this was already observed by~\cite{Nunez-Queija01:0}): the beginning of a new excursion at level $a$ corresponds to an arrival of a customer when there were $a-1$ customers in the queue. The excursion above level $a$ lasts until the queue length process falls back to level $a-1$ and because of the pre-emptive LIFO service discipline, this excursion is distributed as a regular busy cycle. In particular, it is independent of $a$ and moreover, the service discipline is such that successive excursions are i.i.d..

Assuming that the queue length process properly scaled converges and that this regenerative property passes to the limit, it becomes very natural that any accumulation point of the queue length process satisfies the continuous version of the regenerative property as defined by~\cite{Weill07:0}. As~\cite{Weill07:0} showed that this regenerative property characterizes the height process of L\'evy trees~\cite{Duquesne02:0}, this provides another explanation as why, as was showed by~\cite{Limic00:0, Limic01:0}, the pre-emptive LIFO queue is the height process of a L\'evy tree.

\subsection{Organization of the paper}

Section~\ref{sec:model:main} introduces basic notation, presents our main result (Theorem~\ref{thm:main}) and discusses more formally the connection with L\'evy trees. Before proceeding to the proof of Theorem~\ref{thm:main} in Section~\ref{sec:proof}, we introduce in Section~\ref{sec:coupling} the coupling of~\cite{Simatos14:1}, and additional notation together with preliminary results in Section~\ref{sec:notation}. We conclude the paper by discussing in Section~\ref{sec:discussion} possible extensions of our results.

\subsection{Acknowledgements} F.\ Simatos would like to thank N.\ Broutin for useful discussions about branching random walks that lead to the results of Appendix A. The authors are also grateful to I.\ Kortchemski for suggesting the current proof of the second bound in~\eqref{eq:GW} which simplified earlier arguments.

\section{Model and main result}
\label{sec:model:main}

\subsection{Model and main result}

Let $\Mcal$ be the set of finite and positive measures on $[0,\infty)$. We equip $\Mcal$ with the vague topology and consider $\Dcal([0,\infty), \Mcal)$ the class of c\`adl\`ag mappings from $[0,\infty)$ to $\Mcal$, which we endow with the Skorohod topology. It is well-known, see for instance~\cite[Section A.2]{Kallenberg02:0}, that $\Mcal$ and $\Dcal([0,\infty), \Mcal)$ with these topologies are Polish spaces.

Let $\zero \in \Mcal$ be the zero measure, $\delta_a$ be the Dirac mass at $a \geq 0$ and $\Mcal_F \subset \Mcal$ be the set of finite point measures, i.e., measures $\nu \in \Mcal$ with finite support and of the form $\nu = \sum_p \varsigma_p \delta_p$ for some integers $\varsigma_p$. For a measure $\nu \in \Mcal$ let $\pi(\nu)$ be the supremum of its support:
\[ \pi(\nu) = \sup \left\{ y \geq 0: \nu([y, \infty)) > 0 \right\} \]
with the convention $\pi(\zero) = 0$; $\pi(\nu)$ will be called the \emph{price} of the measure $\nu$, and an atom of $\nu \in \Mcal_F$ will be referred to as an \emph{order}.
\\

We use the canonical notation and denote by $(X_t, t \geq 0)$ the canonical $\Mcal$-valued process. Let $\P_\chi$ be the law of the $\Mcal_F$-valued (strong) Markov process started at $\chi \in \Mcal_F$ and with generator $\omega$ given by
\[ \omega(f)(\nu) = \lambda \E \left[ f(\nu + \delta_{(\pi(\nu) + J)^+}) - f(\nu) \right] + \lambda \left[ f(\nu - \delta_{\pi(\nu)}) - f(\nu) \right] \indicator{\nu \not = \zero} \]
where $a^+ = \max(0,a)$ for $a \in \R$, and where $\lambda > 0$ and $J$, a real-valued random variable, are the only two parameters of the model under consideration. In words, the dynamic is as follows. We are given two independent Poisson processes, each of intensity $\lambda$. When the first one rings, a new order is added to the process and is located at a distance distributed like $J$ to the current price, independently from everything else ($J$ will sometimes be referred to as the \emph{displacement} of the newly added order). Note however that an order cannot be placed in the negative half-line, and so an order with displacement $J$ is placed at $(\pi(\nu) + J)^+$ (this boundary condition will be discussed in Section~\ref{sec:discussion}). When the second Poisson process rings and provided that at least one order is present, an order currently sitting at the price is removed (it does not matter which one).
\\

Let $\P^n_\chi$ be the law of $(\vartheta_n(X_{n^2 t}), t \geq 0)$ under $\P_\chi$, where $\vartheta_n: \Mcal \to \Mcal$ acts on measures as follows:
\begin{equation} \label{eq:vartheta}
	\vartheta_n(\nu)([y,\infty)) = \frac{1}{n} \nu([ny, \infty)), \ y \geq 0.
\end{equation}

In the sequel we will omit the subscript when the initial state is the empty measure $\zero$, i.e., we will write $\P$ and $\P^n$ for $\P_\zero$ and $\P^n_\zero$, respectively, with their corresponding expectations $\E$ and $\E^n$. For convenience we will also use $\P$ and $\E$ to denote the probability and expectation of other generic random variables (such as when we write $\E(J)$, or when we consider random trees).

Let $\Mcal^n_F = \vartheta_n(\Mcal_F) = \{ \vartheta_n(\nu): \nu \in \Mcal_F\}$. In the sequel we will denote by $\nu^n$ for $\nu \in \Mcal_F^n$ the only measure in $\Mcal_F$ such that $\vartheta_n(\nu^n) = \nu$. Let in the sequel $W$ be a standard Brownian motion reflected at $0$ and $\alpha = (2 \lambda)^{1/2}$. The following result, which is the main result of the paper, shows that $\P^n$ converges weakly to a measure-valued process which can simply be expressed in terms of $W$.

\begin{theorem} \label{thm:main}
	Assume that $\E(J) > 0$ and that $J \in \left\{-j^*, -j^* + 1, \ldots, 0, 1\right\}$ for some $j^* \in \N$. Then as $n \to +\infty$, $\P^n$ converges weakly to the unique probability measure under which $X$ satisfies the following two properties:
	\begin{enumerate}
		\item[a.] $\pi \circ X$ is equal in distribution to $\alpha \E(J) W$;
		\item[b.] $X_t$ for each $t \geq 0$ is absolutely continuous with respect to Lebesgue measure with density $\indicator{0 \leq y \leq \pi(X_t)} / \E(J)$, i.e.,
		\begin{equation} \label{eq:limit}
			X_t([0,y]) = \frac{1}{\E(J)} \min \left( y, \pi(X_t) \right), \ t, y \geq 0.
		\end{equation}
	\end{enumerate}
\end{theorem}

\begin{rk}
	We will prove more than is stated, namely, we will show that $X$ converges jointly with its mass and price processes, and also with their associated local time processes at $0$ (see Lemma~\ref{lemma:joint-convergence}).
\end{rk}

In the rest of the paper we assume that the assumptions of this theorem hold, i.e., $\E(J) > 0$ and $J \in \{-j^*, \ldots, 1\}$ for some $j^* \in \N$. The behavior when $\E(J) < 0$ is completely different and has been treated in~\cite{Lakner16:0} using stochastic calculus arguments, see the Introduction and Section~\ref{sec:discussion} for more details.

\subsection{Link with L\'evy trees: detailed discussion} \label{sub:link-levy-trees}

The following lemma is at the heart of our approach to prove Theorem~\ref{thm:main}. Let in the sequel $\Dcal$ be the set of real-valued c\`adl\`ag functions with domain $[0,\infty)$ and
\[ \zeta(f) = \inf\{ t > 0: f(t) = 0 \} \]
for $f \in \Dcal$. We call \emph{excursion}, or \emph{excursion away from $0$}, a function $f \in \Dcal$ with $0 < \zeta(f) < + \infty$ and $f(\zeta(f) + t) = 0$ for all $t \geq 0$ (note that we only consider excursions with finite length). We call \emph{height} of an excursion its supremum, and denote by $\Ecal$ the set of excursions. For $a \geq 0$ and $g \leq d$ we say that the function $e = (f((g + t) \wedge d) - a, t \geq 0)$ is an \emph{excursion of $f$ above level $a$} if $e \in \Ecal$, $e_t \geq 0$ for every $t \geq 0$ and $f(g-) \leq a$.

\begin{lemma} \label{lemma:discrete-reg}
	Let $a \geq 0$ be any integer. Then under $\P$, the sequence of successive excursions of $\pi \circ X$ above level $a$ are i.i.d., with common distribution the first excursion of $\pi \circ X$ away from $0$ under $\P_{\delta_1}$.
\end{lemma}

\begin{proof}
	Consider $X$ under $\P_\nu$ for any $\nu \in \Mcal_F$ with $\pi(\nu) \leq a$ that only puts mass on integers. Then when the first excursion (of $\pi \circ X$) above $a$ begins, the price is at $a$ and an order is added at $a+1$. Thus if $g$ is the left endpoint of the first excursion above $a$, $X_g$ must be of the form $X_g = X_{g-} + \delta_{a+1}$ with $\pi(X_{g-}) = a$. This excursion lasts as long as at least one order sits at $a+1$, and if $d$ is the right endpoint of the first excursion above $a$, then what happens during the time interval $[g,d]$ above $a$ is independent from $X_{g-}$ and is the same as what happens above $0$ during the first excursion of $\pi \circ X$ away from $0$ under $\P_{\delta_1}$. Moreover, $X_d$ only puts mass on integers and satisfies $\pi(X_d) \leq a$, so that thanks to the strong Markov property we can iterate this argument. The result therefore follows by induction.
\end{proof}

\begin{rk}
	For $a \geq 0$ let $R_a : \Mcal \to \Mcal$ be defined by
	\[ R_a(\nu)([y,\infty)) = \nu([a+y, \infty)), \]
	and call $(R_a(X_t), g \leq t \leq d)$ an excursion of $X$ above level $a$ if the path $(\pi(X_t), g \leq t \leq d)$ is an excursion of $\pi \circ X$ above level $a$. Then the above proof actually shows that the successive excursions above level $a$ of $X$ are i.i.d., with common distribution the first excursion above $0$ of $X$ under $\P_{\delta_1}$.
\end{rk}

Lemma~\ref{lemma:discrete-reg} is at the heart of our proof of Theorem~\ref{thm:main}. Indeed, this regenerative property is strongly reminiscent of Galton--Watson branching processes. More precisely, consider a stochastic process $H \in \Ecal$ with finite length and continuous sample paths, that starts at $1$, increases or decreases with slope $\pm 1$ and only changes direction at integer times.

For integers $a \geq 0$ and $p > 0$ and conditionally on $H$ having $p$ excursions above level $a$, let $(e^k_{a,p}, k = 1, \ldots, p)$ be these $p$ excursions. Then $H$ is the contour function of a Galton--Watson tree if and only if for each $a$ and $p$, the $(e^k_{a,p}, k = 1, \ldots, p)$ are i.i.d.\ with common distribution $H$. Indeed, $H$ can always be seen as the contour function of some discrete tree. With this interpretation, the successive excursions above $a$ of $H$ code the subtrees rooted at nodes at depth $a+1$ in the tree. The $(e^k_{a,p}, k = 1, \ldots, p)$ being i.i.d.\ therefore means that the subtrees rooted at nodes at depth $a$ are i.i.d.: this is precisely the definition of a Galton--Watson tree.
\\

The difference between this regenerative property and the regenerative property satisfied by $\pi \circ X$ under $\P$ and described in Lemma~\ref{lemma:discrete-reg} is that, \emph{when conditioned to belong to the same excursion away from $0$}, consecutive excursions of $\pi \circ X$ above some level are neither independent, nor identically distributed. If for instance we condition some excursion above level $a$ to be followed by another such excursion within the same excursion away from $0$, this biases the number of orders put in $\{0, \ldots, a\}$ during the first excursion above $a$. Typically, one may think that more orders are put in $\{0,\ldots,a\}$ in order to increase the chance of the next excursion above $a$ to start soon, i.e., before the end of the current excursion away from $0$.

However, this bias is weak and will be washed out in the asymptotic regime that we consider. Thus it is natural to expect that $\pi \circ X$ under $\P$, properly renormalized, will converge to a process satisfying a continuous version of the discrete regenerative property satisfied by the contour function of Galton--Watson trees.
\\

Such a regenerative property has been studied in~\cite{Weill07:0}, who has showed that it characterizes the contour process of L\'evy trees (see for instance~\cite{Duquesne02:0} for a background on this topic). Thus upon showing that this regenerative property passes to the limit, we will have drastically reduced the possible limit points, and it will remain to show that, among the contour processes of L\'evy trees, the limit that we have is actually a reflected Brownian motion. From there, a argument based on local time considerations allows us to conclude that Theorem~\ref{thm:main} holds.

In summary, our proof of Theorem~\ref{thm:main} will be divided into four main steps:
\begin{enumerate}
	\item showing tightness of $\P^n$;
	\item showing, based on Lemma~\ref{lemma:discrete-reg}, that for any accumulation point $\Pbf$, $\pi \circ X$ under $\Pbf$ satisfies the regenerative property studied in~\cite{Weill07:0} (most of the proof is devoted to this point);
	\item arguing that among the contour processes of L\'evy trees, $\pi \circ X$ under $\Pbf$ must actually be a reflected Brownian motion;
	\item showing that $X_t$ under $\Pbf$ has density $\indicator{y \leq \pi(X_t)} / \E(J)$ with respect to Lebesgue measure.
\end{enumerate}

\section{Coupling with a branching random walk} \label{sec:coupling}

\subsection{Coupling with a branching random walk}

In this section we introduce the coupling of~\cite{Simatos14:1} between our model and a particular random walk with a barrier. As mentioned in the Introduction, this coupling plays a crucial role in the proof of Theorem~{\protect\ref{thm:main}}. Let $\T$ be the set of \emph{colored, labelled, rooted} and \emph{ordered} trees. Trees in $\T$ are endowed with the lexicographic order. Thus in addition to its genealogical structure, each edge of a tree $\tree \in \T$ has a real-valued label and each node has one of three colors: either white, green or red.

In the sequel we write $v \in \tree$ to mean that $v$ is a node of $\tree$, and we denote by $\emptyset \in \tree$ the root of $\tree$, by $\lvert \tree \rvert$ its size (the total number of nodes) and by $h(\tree)$ its height. Nodes inherit labels in the usual way, i.e., the root has some label and the label of a node that is not the root is obtained recursively by adding to the label of its parent the label on the edge between them. If $v \in \tree$ we write $\psi(v, \tree)$ for the label of $v$ (in $\tree$), $\lvert v \rvert$ for the depth of $v$ (so that, by our convention, $\lvert \emptyset \rvert = 1$ and $h(\tree) = \sup_{v \in \tree} \lvert v \rvert$) and $v_k \in \tree$ for $k = 1, \ldots, \lvert v \rvert$ for the node at depth $k$ on the path from the root to $v$ (so that $v_1 = \emptyset$ and $v_{\lvert v \rvert} = v$). Also, $\psi^*(\tree) = \sup_{v \in \tree} \psi(v, \tree)$ is the largest label in $\tree$, $\gamma(\tree)$ is the green node in $\tree$ with largest label, with $\gamma(\tree) = \emptyset$ if $\tree$ has no green node and in case several nodes have the largest label, $\gamma(\tree)$ is the last one, and $\Gamma(\tree) \in \Mcal_F$ is the point measure that records the labels of green nodes:
\[ \Gamma(\tree) = \sum_{v \in \tree: v \text{ is green}} \delta_{\psi(v, \tree)}. \]

We say that a node $v \in \tree$ is \emph{killed} if the label of $v$ is $<$ than the label of the root, and if the label of every other node on the path from the root to $v$ has a label $\geq$ to the one of the root. Let $\Kcal(\tree) \subset \tree$ be the set of \emph{killed nodes}:
\[ \Kcal(\tree) = \left\{ v \in \tree: \min_{1 \leq k \leq \lvert v \rvert-1} \psi(v_k, \tree) \geq \psi(\emptyset, \tree) \ \text{ and } \ \psi(v, \tree) < \psi(\emptyset, \tree) \right\} \]
and consider $B(\tree) \in \T$ the tree obtained from $\tree$ by removing all the descendants of the killed nodes (but keeping the killed nodes themselves), and $B_+(\tree)$ the tree obtained from $B(\tree)$ by applying the map $x \mapsto x^+$ to the label of every node in $B(\tree)$. Note that since $B(\tree)$ is a subtree of $\tree$, we always have $\psi^*(B(\tree)) \leq \psi^*(\tree)$.
\\

Let $\Phi: \T \to \T$ be the operator acting on a tree $\tree \in \T$ as follows. If $\tree$ has no green node then $\Phi(\tree) = \tree$. Else, $\Phi$ changes the color of one node in $\tree$ according to the following rule:
\begin{itemize}
	\item if $\gamma(\tree)$ has at least one white child, then its first white child becomes green;
	\item if $\gamma(\tree)$ has no white child, then $\gamma(\tree)$ becomes red.
\end{itemize}

Let $\Phi_k$ be the $k$th iterate of $\Phi$, i.e., $\Phi_0$ is the identity map and $\Phi_{k+1} = \Phi \circ \Phi_k$, and let also $\tau(\tree) = \inf\{ k \geq 0: \psi(\gamma(\Phi_k(\tree)), \tree) < \psi(\emptyset, \tree) \}$. We will sometimes refer to the process $(\Phi_k(\tree), k = 0, \ldots, \tau(\tree))$ as the \emph{exploration of the tree $\tree$}.

Consider a tree $\tree \in \T$ such that all the nodes are white, except for the root which is green. For such a tree, the dynamic of $\Phi$ is such that $\tau(\tree)$ is the smallest $k$ at which the nodes of $B(\Phi_k(\tree)) \setminus \Kcal(\Phi_k(\tree))$ are red, the nodes of $\Kcal(\Phi_k(\tree))$ are green and the other nodes are still white. It has taken one iteration of $\Phi$ to make the nodes of $\Kcal(\tree)$ green, and two to make the nodes of $B(\tree) \setminus \Kcal(\tree)$ red (first each of them had to be made green), except for the root which was already green to start with. Thus for such a tree we have $\tau(\tree) = 2 \lvert B(\tree) \rvert - \lvert \Kcal(\tree) \rvert - 1$.
%%%%%%%%%%%%%%%%%%%
% SKETCH OF PROOF %
%%%%%%%%%%%%%%%%%%%
% Let $\Sigma_k$ be the number of green nodes plus twice the number of red nodes in $\Phi_k(\tree)$, so that $\Sigma_0 = 1$. Consider some $k \geq 0$ such that $\Phi_k(\tree)$ has at least one green node. Then applying $\Phi$: either turns a white node into a green one; or turns a green node into a red one. In each case we have $\Sigma_{k+1} = 1 + \Sigma_k$. In particular, since $\Phi_k(\tree)$ has at least one green node for $k < \tau(\tree)$, we have $\Sigma_{\tau(\tree)} = \tau(\tree)+1$.
%
% Moreover, at time $\tau(\tree)$ all nodes of $B(\tree) \setminus \Kcal(\tree)$ are red and all nodes of $\Kcal(\tree)$ are green, we obtain by definition of $\Sigma$ that $\Sigma_{\tau(\tree)} = \lvert \Kcal(\tree) \rvert + 2 (\lvert B(\tree) \rvert - \lvert \Kcal(\tree) \rvert)$ which gives the result.
%%%%%%%%%%%%%%%%%%%%%%%
% END SKETCH OF PROOF %
%%%%%%%%%%%%%%%%%%%%%%%
\\

Let finally $\Tree_x$ for $x \in \R$ be the following random tree:
\begin{itemize}
	\item its genealogical structure is a (critical) Galton--Watson tree with geometric offspring distribution with parameter $1/2$, i.e., each node has $k = 0, 1, \ldots$ children with probability $1/2^{k+1}$ independently from everything else;
	\item $\psi(\emptyset, \Tree_x) = x$ and labels on the edges are i.i.d., independent from the genealogical structure, and with common distribution $J$;
	\item all nodes are white, except for the root which is green.
\end{itemize}

Because of the last property and the preceding remark, we have
\begin{equation} \label{eq:formula-tau}
	\tau(\Tree_x) = 2 \lvert B(\Tree_x) \rvert - \lvert \Kcal(\Tree_x) \rvert - 1.
\end{equation}
Note that since $J \leq 1$, we have $\psi^*(\Tree_1) \leq h(\Tree_1)$, which gives in particular $\psi^*(B(\Tree_1)) \leq h(\Tree_1)$. The following result is a slight variation of Theorem~$2$ in~\cite{Simatos14:1}, where the same model in discrete-time and without the boundary condition (i.e., an order may be added in the negative half-line) was studied. The intuition behind this coupling is to create a genealogy between orders in the book, a newly added order being declared the child of the order corresponding to the current price, see Section~$3.1$ in~\cite{Simatos14:1} for more details.

\begin{theorem} [Theorem~$2$ in~\cite{Simatos14:1}] \label{thm:coupling}
	Let $a$ be any integer and $g < d$ be the endpoints of the first excursion of $\pi \circ X$ above level $a$. Then the process $(X_t - X_{g-}, g \leq t \leq d)$ under $\P$ and embedded at jump epochs is equal in distribution to the process $(\Gamma \circ \Phi_k \circ B_+(\Tree_{a+1}), k = 0, \ldots, \tau(\Tree_{a+1}))$.
\end{theorem}

\subsection{Ambient tree}

Thanks to this coupling, we can see any piece of path of $X$ corresponding to an excursion of the price process above some level $a$ as the exploration of some random tree $\Tree_{a+1}$: we will sometimes refer to this tree as the \emph{ambient tree}. Note that the ambient tree of an excursion above $a$, say $e$, is a subtree of the ambient tree of the excursion above $a-1$ containing $e$. Moreover, the remark following Lemma~\ref{lemma:discrete-reg} implies that the ambient trees corresponding to successive excursions above some given level are i.i.d..

\subsection{Exploration time} Theorem~\ref{thm:coupling} gives, via~\eqref{eq:formula-tau}, the number of steps needed to explore the ambient tree, say $\Tree$. However, we are interested in $X$ in continuous time. Since jumps in $X$ under $\P$ occur at rate $2 \lambda$ independently from everything else, the length of the corresponding excursion is given by $\Scal(\tau(\Tree))$, where, here and in the sequel, $\Scal$ is a random walk with step distribution the exponential random variable with parameter $2 \lambda$, independent from the ambient tree $\Tree$.

More generally, we will need to control the time needed to explore certain regions of $\Tree$, which will translate to controlling $\Scal(\beta)$ for some random times $\beta$ defined in terms of $\Tree$, and thus independent from $\Scal$. As it turns out, the random variables $\beta$ that need be considered have a heavy tail distribution. Since on the other hand jumps of $\Scal$ are light-tailed, the approximation $\P(\Scal(\beta) \geq y) \approx \P(\beta \geq 2 \lambda y)$ will accurately describe the situation. Let us make this approximation rigorous: for the upper bound, we write
\begin{multline*}
	\P \left( \Scal(\beta) \geq y \right) \leq \P \left( \beta \geq \lambda y \right) + \P \left( \Scal(\beta) \geq y, \beta \leq \lambda y \right)\\
	\leq \P \left( \beta \geq \lambda y \right) + \P \left( \Scal(\lambda y) \geq y \right).
\end{multline*}
Then, a large deviations bound shows that $\P(\Scal(y) \geq y) \leq e^{-\overline \mu y}$ where we have defined $\overline \mu = (1 - \log 2) \lambda$. Carrying out a similar reasoning for the lower bound, we get
\begin{equation} \label{eq:bounds-exploration}
	\P \left( \beta \geq 4 \lambda y \right) - e^{-\underline \mu y} \leq \P \left( \Scal(\beta) \geq y \right) \leq \P \left( \beta \geq \lambda y \right) + e^{-\overline \mu y}
\end{equation}
with $\underline \mu = (2 \log 2 - 1) \lambda$.

\section{Notation and preliminary remarks} \label{sec:notation}

\subsection{Additional notation and preliminary remarks} \label{sub:remarks+notation}

We will write in the sequel $\P_x$, $\P$, $\P_x^n$, $\P^n$ for $\P_{\delta_x}$, $\P_\zero$, $\P_{\delta_x}^n$ and $\P_\zero^n$, respectively, and denote by $\E_x, \E$, etc, the corresponding expectations. Remember that we will also use $\P$ and $\E$ to denote the probability and expectation of other generic random variables (such as when we write $\E(J)$). In the sequel it will be convenient to consider some arbitrary probability measure $\Pbf$ on $\Dcal([0,\infty),\Mcal)$ and to write $Y_n \Rightarrow^n Y$ to mean that the law of $Y_n$ under $\P^n$ converges weakly to the law of $Y$ under $\Pbf$ ($Y_n$ and $Y$ are measurable functions of the canonical process). When we will have proved the tightness of $\P^n$, then we will fix $\Pbf$ to be one of its accumulation points, but until then $\Pbf$ remains arbitrary. Let $M(\nu)$ for $\nu \in \Mcal$ be the mass of $\nu$, i.e., $M(\nu) = \nu([0,\infty))$. If $\phi: [0,\infty) \to [0,\infty)$ is continuous, we will denote by $f_\phi: \Mcal \to [0,\infty]$ the function defined for $\nu \in \Mcal$ by $f_\phi(\nu) = \int \phi \d \nu$.

We will need various local time processes at $0$. Let $\ell_t = \int_0^t \indicator{\pi(X_u) = 0} \d u$ denote the Lebesgue measure of the time spent by the price process at $0$. For discrete processes, i.e., under $\P^n$, we will also need the following local time processes at $0$ of $M \circ X$ and $\pi \circ X$:
\[ L^{n, M}_t = n \int_0^t \indicator{M(X_u) = 0} \d u \ \text{ and } \ L^{n, \pi}_t = n \int_0^t \indicator{\pi(X_u) = 0} \d u, \ t \geq 0, n \geq 1. \]

For the continuous processes that will arise as the limit of $\pi \circ X$ and $M \circ X$, we consider the operator $\Lcal$ acting on continuous functions $f: [0,\infty) \to [0,\infty)$ as follows:
\[ \Lcal(f)_t = \lim_{\varepsilon \downarrow 0} \left( \frac{1}{\varepsilon} \int_0^t \indicator{f(u) \leq \varepsilon} \d u \right), \ t \geq 0. \]

We will only consider $\Lcal$ applied at random processes equal in distribution to $\beta W$ for some $\beta > 0$, in which case this definition makes sense and indeed leads to a local time process at $0$. Note that for any $\beta > 0$ and any $f$ for which $\Lcal(f)$ is well-defined, we have $\Lcal(\beta f) = \beta^{-1} \Lcal(f)$. Moreover, according to Tanaka's formula the canonical semimartingale decomposition of $W$ is given by
\begin{equation} \label{eq:canonical-decomposition-W}
	W = \frac{1}{2} \Lcal(W) + \bar W
\end{equation}
where $\bar W$ is a standard Brownian motion.
\\

In the sequel we will repeatedly use the fact that the process $\pi \circ X$ under $\P^n$ (or $\P$) is regenerative at $0$, in the sense that successive excursions away from $0$ are i.i.d.. Note also that the time durations between successive excursions away from $0$ are also i.i.d., independent from the excursions, with common distribution the exponential random variable (with parameter $\lambda \P(J = 1) n^2$ under $\P^n$, and $\lambda \P(J=1)$ under $\P$).

Moreover, jumps of $\pi \circ X$ under $\P^n$ have size $1/n$, and so if $\pi \circ X$ under $\P^n$ converges weakly, then the limit must be almost surely continuous (see for instance Theorem~$13.4$ in~\cite{Billingsley99:0}).
\\

Let $\theta_t$ and $\sigma_t$ for $t \geq 0$ be the shift and stopping operators associated to $\pi \circ X$, i.e., $\theta_t = (\pi(X_{t + s}), s \geq 0)$ and $\sigma_t = (\pi(X_{s \wedge t}), s \geq 0)$. Since by the previous remark, accumulation points of $\pi \circ X$ under $\P^n$ are continuous, these operators are continuous in the following sense (see for instance~\cite[Lemma~$2.3$]{Lambert13:0}).

\begin{lemma} [Continuity of the shift and stopping operators] \label{lemma:continuity-shift+stopping}
	Consider some arbitrary random times $T^n, T \geq 0$. If $(\pi \circ X, T^n) \Rightarrow^n (\pi \circ X, T)$, then $(\theta_{T^n}, \sigma_{T^n}) \Rightarrow^n (\theta_T, \sigma_T)$.
\end{lemma}

We will finally need various random times. For $t$ and $\varepsilon \geq 0$ let
\[ G_t = \sup \left\{ s \leq t: \pi(X_s) = 0 \right\}, \ D_t = \inf\left\{ s \geq t: \pi(X_s) = 0 \right\} \]
and
\[ D_{t, \varepsilon} = \inf\left\{ s \geq t: \pi(X_s) \leq \varepsilon \right\}. \]

Note that $G_t$ and $D_t$ are the endpoints of the excursion of $\pi \circ X$ straddling $t$, where we say that an excursion straddles $t$ if its endpoints $g \leq d$ satisfy $g \leq t \leq d$. For $0 \leq a \leq b$ we also define $T_b = \inf \{ s \geq 0: \pi(X_s) \geq b \}$ and
\begin{equation} \label{eq:def-g-d}
	g_{a, b} = \sup \left\{ s \leq T_b: \pi(X_s) = a \right\}, \ d_{a,b} = \inf \left\{ s \geq T_b: \pi(X_s) = a \right\}
\end{equation}
and
\begin{equation} \label{eq:def-U}
	U_{a,b} = d_{a,b} - g_{a,b},
\end{equation}
so that $g_{a,b} \leq d_{a,b}$ are the endpoints of the first excursion of $\pi \circ X$ above level $a$ with height $\geq b-a$ and $U_{a,b}$ is its length. Note that, in terms of trees, the interval $[g_{a,b}, d_{a,b}]$ corresponds to the exploration of a tree distributed like $\Tree_a$ conditioned on $\psi^*(\Tree_a) > b$, since the height of the excursion corresponds to the largest label in the ambient tree. Also, it follows from the discussion at the end of Section~\ref{sec:coupling} that $U_{a,b}$ is equal in distribution to $\Scal(\tau(\Tree_a))$ under the same conditioning.

\subsection{An aside on the convergence of random times} \label{sub:aside-hitting-times}

At several places in the proof of Theorem~\ref{thm:main} it will be crucial to control the convergence of some specific random times. For instance, we will need to show in the fourth step of the proof that if $(X, \pi \circ X) \Rightarrow^n (X, \pi \circ X)$, then $D_t \Rightarrow^n D_t$ for any $t \geq 0$. Let us explain why, in order to show that $D_t \Rightarrow^n D_t$, it is enough to show that for any $\eta > 0$,
\begin{equation} \label{eq:hitting-times}
	\limsup_{n \to +\infty} \P^n \left( D_t - D_{t, \varepsilon} \geq \eta \right) \mathop{\longrightarrow}_{\varepsilon \to 0} 0.
\end{equation}

Let us say that $\pi \circ X$ \emph{goes across} $\varepsilon$ if $\inf_{[D_{t, \varepsilon}, D_{t, \varepsilon} + \eta]} \pi \circ X < \varepsilon$ for every $\eta > 0$, and let $\Gcal = \{ \varepsilon > 0: \pi \circ X \text{ goes across } \varepsilon \}$. Then, the following property holds (see for instance Proposition~VI.$2$.$11$ in~\cite{Jacod03:0} or Lemma~$3.1$ in~\cite{Lambert15:0}): if $\Pbf(\varepsilon \in \Gcal) = 1$, then $D_{t, \varepsilon} \Rightarrow^n D_{t, \varepsilon}$.

On the other hand, the complement $\Gcal^c$ of $\Gcal$ is precisely the set of discontinuities of the process $(D_{t, \varepsilon}, \varepsilon > 0)$. Since $(D_{t, \varepsilon}, \varepsilon > 0)$ is c\`agl\`ad because it is the left-continuous inverse of the process $(\inf_{[t,t+s]} \pi \circ X, s \geq 0)$, the set $\{ \varepsilon > 0: \Pbf(\varepsilon \in \Gcal^c) > 0\}$ is at most countable, see for instance~\cite[Section $13$]{Billingsley99:0}. Gathering these two observations, we see that the convergence $D_{t, \varepsilon} \Rightarrow^n D_{t, \varepsilon}$ holds for all $\varepsilon > 0$ outside a countable set. Then, writing for any $\varepsilon, \eta > 0$
\[ \P^n \left( D_t \geq x \right) = \P^n \left( D_t \geq x, D_t - D_{t, \varepsilon} \geq \eta \right) + \P^n \left( D_t \geq x, D_t - D_{t, \varepsilon} < \eta \right) \]
gives
\[ \P^n \left( D_t \geq x \right) \leq \P^n \left( D_t - D_{t, \varepsilon} \geq \eta \right) + \P^n \left( D_{t, \varepsilon} \geq x - \eta \right). \]
Since $D_{t, \varepsilon} \Rightarrow^n D_{t, \varepsilon}$ for all $\varepsilon$ outside a countable set, and since for those $\varepsilon$ we have $\P^n \left( D_{t, \varepsilon} \geq x - \eta \right) \to \Pbf \left( D_{t, \varepsilon} \geq x - \eta \right)$ for all $\eta$'s outside a countable set, we obtain for all $\varepsilon, \eta > 0$ outside a countable set
\[ \limsup_{n \to +\infty} \P^n \left( D_t \geq x \right) \leq \limsup_{n \to +\infty} \P^n \left( D_t - D_{t, \varepsilon} \geq \eta \right) + \Pbf \left( D_{t, \varepsilon} \geq x - \eta \right). \]

Next, observe that $D_{t, \varepsilon} \to D_t$ as $\varepsilon \to 0$, $\Pbf$-almost surely. Indeed, $D_{t, \varepsilon}$ decreases as $\varepsilon \downarrow 0$, and its limit $D'$ must satisfy $t \leq D' \leq D_t$ because $t \leq D_{t, \varepsilon} \leq D_t$, and also $\pi(X_{D'}) = 0$ because $\pi(X_{D_{t, \varepsilon}}) \leq \varepsilon$ and $\pi \circ X$ is $\Pbf$-almost surely continuous. Thus letting first $\varepsilon \to 0$ and then $\eta \to 0$ in the previous display, we obtain by~\eqref{eq:hitting-times}
\[ \limsup_{n \to +\infty} \P^n \left( D_t \geq x \right) \leq \Pbf \left( D_t \geq x \right) \]
which shows that $D_t \Rightarrow^n D_t$ by the Portmanteau theorem. This reasoning, detailed for $D_t$ and used in the proof of Lemma~\ref{lemma:D}, will also be used in Section~\ref{sub:R} to control the asymptotic behavior of $T_b$, $g_{a,b}$ and $d_{a,b}$.
\\

We will also use the following useful property: if $\pi \circ X$ and $D_t$ converge weakly, then the convergence actually holds jointly. The reasoning goes as follows. If $\pi \circ X$ and $D_t$ under $\P^n$ converge to $\pi \circ X$ and $D_t$ under $\Pbf$, then $(\pi \circ X, D_t)$ under $\P^n$ is tight (we always consider the product topology). Let $(P', D')$ be any accumulation point.

Since projections are continuous, $P'$ is equal in distribution to $\pi \circ X$ under $\Pbf$, in particular it is almost surely continuous, and $D'$ is equal in distribution to $D_t$ under $\Pbf$, in particular it is almost surely ${\geq} t$. Further, assume using Skorohod's representation theorem that $(P^n, D^n_t)$ is a version of $(\pi \circ X, D_t)$ under $\P^n$ which converges almost surely to $(P', D')$. Since $P^n_{D^n_t} = 0$ and $P'$ is continuous, we get $P'_{D'} = 0$ and thus, since $D' \geq t$, $\inf\{s \geq t: P'_s = 0\} \leq D'$. Since these two random variables are both equal in distribution to $D_t$ under $\Pbf$, they must be (almost surely) equal. This shows that $(P', D')$ is equal in distribution to $(\pi \circ X, D_t)$ under $\Pbf$, which uniquely identifies accumulation points.

This reasoning applies to all the random times considered in this paper, in particular to $T_b$, $g_{a,b}$ and $d_{a,b}$. Thus, once we will have shown the convergence of $\pi \circ X$ and, say, $T_b$, then we will typically be in position to use Lemma~\ref{lemma:continuity-shift+stopping} and deduce the convergence of $\theta_{T_b}$ and $\sigma_{T_b}$.

\subsection{Convention} In the sequel we will need to derive numerous upper and lower bounds, where only the asymptotic behavior up to a multiplicative constant matters. It will therefore be convenient to denote by $C$ a strictly positive and finite constant that may change from line to line, and even within the same line, but which is only allowed to depend on $\lambda$ and the law of $J$.

\section{Proof of Theorem~\ref{thm:main}} \label{sec:proof}

We decompose the proof of Theorem~\ref{thm:main} into several steps. The coupling of Theorem~\ref{thm:coupling} makes it possible to translate many questions on $\P^n$ to questions on $B(\Tree_1)$, and in order to keep the focus of the proof on $\P^n$, we postpone to the Appendix~\ref{appendix} the proofs of the various results on $B(\Tree_1)$ which we need along the way.

At a high level, it is useful to keep in mind that, since $\E(J) > 0$, the law of large numbers prevails and the approximation $\psi(v, \Tree_1) \approx \E(J) \lvert v \rvert$ describes accurately enough (for our purposes) the labels in the tree $B(\Tree_1)$. In some sense, most of the randomness of $B(\Tree_1)$ lies in its genealogical structure, and the results of the Appendix~\ref{appendix} aim at justifying this approximation.

Note that similar results than the ones we need here are known in a more general setting, but for the tree without the barrier, i.e., for $\Tree_1$ instead of $B(\Tree_1)$, see, e.g.,~\cite{Durrett91:0} and~\cite{Kesten94:0}.
\\

We begin with a preliminary lemma: recall that $W$ is a reflected Brownian motion, that $\alpha = (2\lambda)^{1/2}$ and that $\Lcal(\beta W) = \beta^{-1} \Lcal(W)$ for any $\beta > 0$.

\begin{lemma} \label{lemma:preliminary}
	As $n \to +\infty$, $(M \circ X, L^{n,M}, L^{n,\pi})$ under $\P^n$ converges weakly to
	\begin{equation} \label{eq:joint-law}
		\left( \alpha W, \Lcal( \alpha W), \Lcal(\alpha \E(J) W) \right).
	\end{equation}
	
	Moreover,
	\begin{equation} \label{eq:bound-M/M/1}
		\sup \, \left\{ \varepsilon^{-1/2} \E^n_{\nu} \left( L^{n,M}_{\varepsilon} \right) : n \geq 1, 0 < \varepsilon < 1, \nu \in \Mcal_F \right\} < +\infty.
	\end{equation}
\end{lemma}

\begin{proof}
	By definition, $M \circ X$ under $\P$ is a critical $M/M/1$ queue with input rate $\lambda$, which is well-known to converge under $\P^n$ to $\alpha W$. Further, $\lambda L^{n,M}$ is the finite variation process that appears in its canonical (semimartingale) decomposition, and standard arguments show that it converges, jointly with $M \circ X$, to the finite variation process that appears in the canonical decomposition of $\alpha W$, equal to $(\alpha/2) \Lcal(W)$ by~\eqref{eq:canonical-decomposition-W}. Dividing by $\lambda$ we see that $L^{n,M}$ under $\P^n$ converges to $(\alpha / (2 \lambda)) \Lcal(W) = \Lcal(\alpha W)$. This shows that $(M \circ X, L^{n,M})$ under $\P^n$ converges weakly to $(\alpha W, \Lcal(\alpha W))$.
	\\
	
	We now show that $L^{n,\pi}$ under $\P^n$ converges weakly to $(1/\E(J)) \Lcal(\alpha W)$ jointly with $M \circ X$ and $L^{n,M}$. Since $L^{n,M}_t / L^{n,\pi}_t$ under $\P^n$ is equal to
	\[ \frac{1}{\ell_{n^2 t}} \int_0^{n^2 t} \indicator{M(X_u) = 0} \d u \]
	under $\P$, it is enough to show that $\int_0^y \indicator{M(X_u) = 0} \d u / \ell_y \to \E(J)$ as $y \to +\infty$, $\P$-almost surely. Indeed, this would imply that $L^{n,\pi}$ under $\P^n$ converges in the sense of finite-dimensional distributions to $(1/\E(J)) \Lcal(\alpha W)$ (jointly with $M \circ X$ and $L^{n,M}$), and so, since $L^{n, \pi}$ and $\Lcal(W)$ are continuous and increasing, Theorem~VI.$2$.$15$ in~\cite{Jacod03:0} would imply the desired functional convergence result.

	Let $Q = M \circ X \circ \ell^{-1}$, where $\ell^{-1}$ stands for the right-continuous inverse of $\ell$. The composition with $\ell^{-1}$ makes $Q$ evolve only when the price is at $0$. Under $\P$ and while the price is at $0$, the dynamic of $Q$ is as follows:
	\begin{itemize}
		\item $Q$ increases by one at rate $\lambda \P(J \leq 0)$ (which corresponds to an order with a displacement $\leq 0$ being added) and decreases by one at rate $\lambda$, provided $Q > 0$ (which corresponds to an order being removed);
		\item when an order with displacement $>0$ is added, which happens at rate $\lambda \P(J = 1)$, the price makes an excursion away from $0$. When it comes back to $0$, $Q$ resumes evolving and, by the coupling, a random number of orders distributed like $\lvert \Kcal(\Tree_1) \rvert$ and independent from everything else have been added at $0$.
	\end{itemize}
	
	Thus we see that $Q$ under $\P$ is stochastically equivalent to a $G/M/1$ single-server queue, with two independent Poisson flows of arrivals: customers arrive either one by one at rate $\lambda \P(J \leq 0)$, or by batch of size distributed according to $\lvert \Kcal(\Tree_1) \rvert$ at rate $\lambda \P(J = 1)$. Then, customers have i.i.d.\ service requirements following an exponential distribution with parameter $\lambda$. In particular, the load of this queue is $\P(J \leq 0) + \P(J=1) \E (\lvert \Kcal(\Tree_1) \rvert)$ which by~\eqref{eq:mean-killed} is equal to $1 - \E(J)$. Since $\E(J) > 0$, $Q$ is positive recurrent and in particular, the long-term average idle time is equal to one minus the load, i.e.,
	\begin{equation} \label{eq:G/M/1}
		\frac{1}{y} \int_0^y \indicator{Q_u = 0} \d u \mathop{\longrightarrow}_{y \to +\infty} \E(J), \ \P-\text{almost surely.}
	\end{equation}
	
	Fix on the other hand some $y > 0$: then
	\begin{align*}
		\int_0^y \indicator{M(X_u) = 0} \d u & = \int_0^y \indicator{M(X_u) = 0}\indicator{\pi(X_u) = 0} \d u\\
		& = \int_0^y \indicator{Q(\ell_u) = 0} \d \ell_u\\
		& = \int_0^{\ell_y} \indicator{Q_u = 0} \d u
	\end{align*}
	which, combined with~\eqref{eq:G/M/1}, proves that $\int_0^y \indicator{M(X_u) = 0} \d u / \ell_y \to \E(J)$ and achieves the proof of the convergence of $(M \circ X, L^{n,M}, L^{n,\pi})$.
	
	It remains to prove~\eqref{eq:bound-M/M/1}: since $M \circ X$ spends more time at $0$ when started empty (this can be easily seen with a coupling argument), we have $\E^n_\nu(L^{n,M}_\varepsilon) \leq \E^n(L^{n,M}_\varepsilon)$ which gives the uniformity in $\nu$. Further, since as mentioned previously $M \circ X - \lambda L^{n,M}$ is a martingale, we have $\E^n(L^{n,M}_\varepsilon) = \lambda^{-1} \E^n( M(X_\varepsilon))$. Since $M \circ X$ is a reflected critical random walk with jump size $\pm 1/n$ and jump rates $\lambda n^2$, one easily proves that $\E^n(M(X_\varepsilon)^2) \leq 2 \lambda \varepsilon$ which gives the desired result by Cauchy-Schwarz inequality.
\end{proof}

\mysubsection{tightness of $\P^n$}

To show the tightness of $\P^n$, it is enough to show that $M \circ X$ under $\P^n$ is tight, and that for each continuous $\phi$ which is infinitely differentiable with a compact support, $f_\phi \circ X$ under $\P^n$ is tight (recall that $f_\phi(\nu) = \int \phi \d \nu$), see for instance Theorem~$2.1$ in~\cite{Roelly-Coppoletta86:0}. The tightness of $M \circ X$ is a direct consequence of Lemma~\ref{lemma:preliminary}, and so it remains to show the tightness of $f_\phi \circ X$. First of all, note that jumps of $f_\phi \circ X$ under $\P^n$ are upper bounded by $\sup \lvert \phi \rvert / n$, and so we only need to control the oscillations of this process (see for instance the Corollary on page~$179$ in~\cite{Billingsley99:0}). Using standard arguments, we see that the process $f_\phi \circ X$ is a special semimartingale with canonical decomposition
\[ \int \phi \d X_t = \int \phi \d X_0 + \int_0^t \Omega_n(f_\phi)(X_u) \d u + Z^n_t, \]
where $\Omega_n$ is the generator of $\P^n$, given for any $\nu \in \Mcal^n_F$ and any function $f: \Mcal \to \R$ by
\begin{multline*}
	\Omega_n(f)(\nu) = \lambda n^2 \E \left[ f\left(\nu + n^{-1} \delta_{(\pi(\nu) + J/n)^+} \right) - f(\nu) \right]\\
	+ \lambda n^2 \left[ f\left(\nu - n^{-1} \delta_{\pi(\nu)} \right) - f(\nu) \right] \indicator{M(\nu) > 0},
\end{multline*}
and $Z^n$ is a local martingale with predictable quadratic variation process given by $\langle Z^n \rangle_t = \int_0^t \Omega^2_n(f_\phi)(X_u) \d u$, with
\begin{multline*}
	\Omega^2_n(f)(\nu) = \lambda n^2 \E \left[ \left( f\left(\nu + n^{-1} \delta_{(\pi(\nu) + J/n)^+} \right) - f(\nu) \right)^2 \right]\\
	+ \lambda n^2 \left[ f\left(\nu - n^{-1} \delta_{\pi(\nu)} \right) - f(\nu) \right]^2 \indicator{M(\nu) > 0}
\end{multline*}
(see, e.g., Lemma~VIII.$3.68$ in~\cite{Jacod03:0}). In particular, we have
\begin{align*}
	\Omega_n(f_\phi)(\nu) & = \lambda n \E \left[ \phi\left((\pi(\nu) + J/n)^+\right) \right] - \lambda n \phi(\pi(\nu)) \indicator{M(\nu) > 0}\\
	& = \lambda n \E \left[ \phi\left((\pi(\nu) + J/n)^+\right) - \phi(\pi(\nu)) \right] + \lambda n \phi(0) \indicator{M(\nu) = 0}
\end{align*}
and
\[ \Omega^2_n(f_\phi)(\nu) = \lambda \E \left[ \phi((\pi(\nu) + J/n)^+)^2\right]\\
- \lambda \phi(\pi(\nu))^2 \indicator{M(\nu) > 0}, \]
from which it follows that
\[ \left \lvert \Omega_n(f_\phi)(\nu) \right \rvert \leq j^* \lambda \sup \lvert \phi' \rvert + n \lambda \lvert \phi(0) \rvert \indicator{M(\nu) = 0} \ \text{ and } \ \left \lvert \Omega_n^2(f_\phi)(\nu) \right \rvert \leq 2 \lambda \sup \phi^2. \]

Thus there exists a finite constant $C'$, that only depends on $\lambda$, the law of $J$ and $\phi$, such that for any finite stopping time $V$ and any $\varepsilon > 0$ we have
\[ \E^n \left( \left \lvert \int_V^{V + \varepsilon} \Omega_n(f_\phi)(X_u) \d u \right \rvert \right) \leq C' \varepsilon + C' \E^n \left( L^{n,M}_{V + \varepsilon} - L^{n,M}_V \right) \leq C' \varepsilon + C' \varepsilon^{1/2}, \]
where the last inequality follows from~\eqref{eq:bound-M/M/1} combined with the strong Markov property at time $V$. Similarly, $\E^n \left( \left \lvert \langle Z^n \rangle_{V + \varepsilon} - \langle Z^n \rangle_V \right \rvert \right) \leq C' \varepsilon$ and these upper bounds imply the tightness of $f_\phi \circ X$ by standard arguments for the tightness of a sequence of semimartingales, see for instance Theorem~VI.$4$.$18$ in~\cite{Jacod03:0}, or Theorem~$2.3$ in~\cite{Roelly-Coppoletta86:0}.
\\

We now know that $\P^n$ is tight: it remains to identify accumulation points. As planned in Section~\ref{sub:remarks+notation}, we now let $\Pbf$ be an arbitrary accumulation point of $\P^n$ and we assume without loss of generality that $\P^n$ converges weakly to $\Pbf$. In particular, we have $f_\phi \circ X \Rightarrow^n f_\phi \circ X$ for every continuous function $\phi \geq 0$ with a compact support, see for instance Theorem~$16.16$ in~\cite{Kallenberg02:0}. Also, as noted in Section~\ref{sub:remarks+notation} the process $\pi \circ X$ under $\Pbf$ is almost surely continuous, and since the jumps of $f_\phi \circ X$ under $\P^n$ are upper bounded by $n^{-1} \sup \lvert \phi \rvert$, the same argument shows that the process $f_\phi \circ X$ under $\Pbf$ is also almost surely continuous.

\mysubsection{joint convergence}

We now show that $X$ under $\P^n$ actually converges jointly with its mass, price and local time processes. The proof is based on checking that some sample-path properties are satisfied, which makes it possible to use the continuous mapping theorem.

\begin{lemma} \label{lemma:joint-convergence}
	The following joint convergence holds:
	\[ \left( X, M \circ X, \pi \circ X, L^{n,M}, L^{n,\pi} \right) \Rightarrow^n \big( X, M \circ X, \pi \circ X, \Lcal(M \circ X), \Lcal(\E(J) M \circ X) \big). \]
	Moreover, it holds $\Pbf$-almost surely that $M(X_t) \geq \pi(X_t)$ for every $t \geq 0$.
\end{lemma}

\begin{proof}
	Lemma~\ref{lemma:preliminary} and the first step imply that the sequence $(X, M \circ X, L^{n,M}, L^{n,\pi})$ under $\P^n$ is tight. Let $(X', M', \Lcal(M'), \Lcal(\E(J) M'))$ be any accumulation point (which is necessarily of this form by Lemma~\ref{lemma:preliminary}), and assume in the rest of the proof, using Skorohod's representation theorem, that $X^{(n)}$ is a version of $X$ under $\P^n$ and that $L^{(n),M}$ and $L^{(n),\pi}$ are defined in terms of $X^{(n)}$ similarly as $L^{n,M}$ and $L^{n,\pi}$ are defined in terms of $X$, such that $(X^{(n)}, M \circ X^{(n)}, L^{(n),M}, L^{(n),\pi}) \to (X', M', \Lcal(M'), \Lcal(\E(J) M'))$ almost surely. Then, in order to prove the joint convergence, we only have to prove that $M' = M \circ X'$ and that $\pi(X^{(n)}_t) \to \pi(X'_t)$ for every $t \geq 0$.
	
	 We will use the following key observation: under $\P$ and provided that $M(X_t) > 0$, we have $X_t(\{p\}) \geq 1$ for any integer $p \leq \pi(X_t)$. Indeed, this is a consequence of our assumption that $J$ is integer-valued with $J \leq 1$ and can be proved by induction. It follows that, $\P^n$-almost surely, $\pi(X_t) \leq M(X_t)$ for every $t \geq 0$ and so $\pi(X^{(n)}_t) \leq M(X^{(n)}_t)$. Note that this implies the desired inequality $\pi(X_t) \leq M(X_t)$ under $\Pbf$, once we will have proved that $M' = M \circ X$ and that $\pi(X^{(n)}_t) \to \pi(X'_t)$.
	
	So fix some $t \geq 0$ and let us first show that $M'_t = M(X'_t)$. Let $K \geq M'_t + 1$ and $\phi$ be any decreasing, continuous function with $\phi(x) = 1$ for $x \leq K$ and $\phi(x) = 0$ for $x \geq K+1$, so that $\int \phi \d X^{(n)}_t \to \int \phi \d X'_t$ as $n \to +\infty$. On the other hand, since $M(X^{(n)}_t) \to M'_t$ we have $M(X^{(n)}_t) \leq M'_t+1$, and in particular $\pi(X^{(n)}_t) \leq K$, for $n$ large enough. Since $\phi(x) = 1$ for $x \leq K$, we have $\int \phi \d X^{(n)}_t = M(X^{(n)}_t)$ for those $n$, and since the left-hand side of this equality converges to $\int \phi \d X'_t$ while the right-hand side converges to $M'_t$, we obtain $\int \phi \d X'_t = M'_t$. Letting $K \to +\infty$ we obtain $M(X'_t) = M'_t$.
		
	Let us now prove that $\pi(X^{(n)}_t) \to \pi(X'_t)$. First, note that since $\pi(X^{(n)}_t) \leq M(X^{(n)}_t)$, the sequence $(\pi(X^{(n)}_t), n \geq 1)$ is bounded and any accumulation point is upper bounded by $M(X'_t)$. Consider any such accumulation point $p$, and assume without loss of generality that $\pi(X^{(n)}_t) \to p \leq M(X'_t)$: we have to show that $p = \pi(X'_t)$. Fix some $b > a > p$ and let $\phi$ be any continuous function with compact support in $[a,b]$: then $\int \phi \d X^{(n)}_t \to \int \phi \d X'_t$ and since $\pi(X^{(n)}_t) \to p$ and $a > p$, we have $\int \phi \d X^{(n)}_t = 0$ for $n$ large enough. Thus $\int \phi \d X'_t = 0$ which shows that $\pi(X'_t) \leq p$. To show the reverse inequality, consider $a < b < p$ such that $X^{(n)}_t([a,b]) \to X'_t([a,b])$ (this holds for every $a < b < p$ outside a countable set). Then for $n$ large enough we have $a < b < \pi(X^{(n)}_t)$ and so a consequence of the key observation made at the beginning of the proof is that $X^{(n)}_t([a,b]) \geq b-a$. This shows that $X'_t([a,b]) \geq b-a$ and so $\pi(X'_t) \geq a$. Letting $a \uparrow p$ achieves the proof.
\end{proof}

Since under $\P^n$, orders are added at a distance at most $j^*/n$ from the current price, it follows readily from the previous result that $X$ under $\Pbf$ only evolves locally around its price, in the sense that if $y \geq 0$ and $0 \leq g \leq d$ are such that $\pi(X_t) > y$ for $g \leq t \leq d$, then for $t \in [g,d]$ the measures $X_t$ and $X_g$ restricted to $[0,y]$ are equal. Actually we will only need the following weaker property.

\begin{corollary} \label{cor:local-evolution}
	The following property holds $\Pbf$-almost surely. Let $t, y \geq 0$ such that $\pi(X_t) > y$, and let $g$ be the left endpoint of the excursion of $\pi \circ X$ above $y$ straddling $t$. Then $X_t([0,y]) = X_g([0,y])$.
\end{corollary}

%%%%%%%%%
% PROOF %
%%%%%%%%%
% \begin{proof}
% 	Assume as in the previous proof that $X^n$ is a version of $X$ under $\P^n$ and $X'$ a version of $X$ under $\Pbf$ such that $(X^n, \pi \circ X^n) \to (X', \pi \circ X')$, almost surely. Let $p < p' < \inf_{[g,d]} \pi \circ X'$ and $\phi$ be any continuous function with compact support included in $[0,p']$. Since $\inf_{[g,d]} \pi \circ X^n \to \inf_{[g,d]} \pi \circ X'$, we have $\inf_{[g,d]} \pi \circ X^n \geq p''$ for $n$ large enough, for any $p' < p'' < \inf_{[g,d]} \pi \circ X'$. Since orders are added to and removed from $X_n$ at distance at most $j^*/n$ from $\pi(X^n_t)$, we see that for $n$ with $\inf_{[g,d]} \pi \circ X^n \geq p''$ we have $\int \phi \d X^n_t = \int \phi \d X^n_g$. Passing to the limit, we obtain $\int \phi \d X'_t = \int \phi \d X'_g$, and the desired result is proved.
% \end{proof}
%%%%%%%%%%%%%
% END PROOF %
%%%%%%%%%%%%%

\mysubsection{$\Pbf$ is regenerative at $0$} \label{sub:proof-regeneration-at-0}

So far, we have used the fact that $\pi \circ X$ under $\P^n$ was regenerative at $0$, in the natural sense that successive excursions away from $0$ are i.i.d.. Under $\Pbf$ there is no first excursion away from $0$ and so we need a more general notion of regeneration (we also don't know, at this point, that $\pi \circ X$ under $\Pbf$ is a Markov process). The goal of this step is to show that $\pi \circ X$ under $\Pbf$ is regenerative at $0$ in the following sense (recall that $\theta_t$ and $\sigma_t$ are the shift and stopping operators associated to $\pi \circ X$, see Section~\ref{sub:remarks+notation}),:
\begin{enumerate}[label=\roman*)]
	\item \label{prop:zero-set} the zero set of $\pi \circ X$ has zero Lebesgue measure under $\Pbf$;
	\item \label{prop:D'} $\Pbf(D'_0 = 0) = 1$, where $D'_0 = \inf\left\{ t > 0: \pi(X_t) = 0 \right\}$;
	\item \label{prop:reg} for every $t \geq 0$ and every continuous, bounded functions $f$ and $g$ on~$\Ecal$:
	\begin{equation} \label{eq:def-continuous-regeneration}
		\Ebf \left[ f(\sigma_{D_t}) g(\theta_{D_t}) \right] = \Ebf \left[ f(\sigma_{D_t}) \right] \Ebf \left[ g(\pi \circ X) \right];
	\end{equation}
\end{enumerate}
Note for~\eqref{eq:def-continuous-regeneration} that $D_t$ is $\Pbf$-almost surely finite (because Lemma~\ref{lemma:joint-convergence} implies that $D_t \leq \inf\{s \geq t: M(X_s) = 0\}$), and this upper bound is $\Pbf$-almost surely finite by Lemma~\ref{lemma:preliminary}. It can be checked, following the proof of Theorem~$22.11$ in~\cite{Kallenberg02:0}, that if $\pi \circ X$ under $\Pbf$ satisfies the three properties~\ref{prop:zero-set}--\ref{prop:reg} above, then $\pi \circ X$ admits an excursion measure away from~$0$, denoted by $\Ncal$, by which we mean, in accordance with the literature, that:
\begin{enumerate}
	\item there is a continuous, nondecreasing process~$L$ increasing only on the zero set of~$\pi \circ X$ (the local time);
	\item the right-inverse~$L^{-1}$ of~$L$ is a subordinator and the excursion process sending~$t$ to the corresponding excursion if~$L^{-1}$ jumps at time~$t$ (and to a cemetery point otherwise) is a Poisson point process with intensity measure $\d \textrm{m} \d \Ncal$ ($\textrm{m}$ stands for the Lebesgue measure).
\end{enumerate}

The rest of this step is devoted to showing that $\Pbf$ satisfies the properties~\ref{prop:zero-set}--\ref{prop:reg} above. We begin with a preliminary lemma.

\begin{lemma} \label{lemma:control-mean-local-time}
	For any $n \geq 1$ and $t \geq 0$, $\E^n(L^{n,\pi}_t) \leq C t^{1/2}$.
\end{lemma}

\begin{proof}
	We have
	\begin{align*}
		\E^n(L^{n,\pi}_t) & = n \E^n \left( \int_0^t \indicator{\pi(X_u) = 0} \d u \right)\\
		& = n \E \left( \int_0^t \indicator{\pi(X_{n^2 u}) = 0} \d u \right)\\
		& = \frac{1}{n} \E \left( \int_0^{n^2 t} \indicator{\pi(X_u) = 0} \d u \right).
	\end{align*}
	
	Let $D^0 = 0$ and $D^k$ for $k \geq 1$ be the endpoint of the $k$th excursion away from $0$ of $\pi \circ X$, and let $K(y) = \sum_{k \geq 1} \indicator{D^k \leq y}$ be the number of excursions finishing before $y$. In particular,
	\[ \int_0^{n^2 t} \indicator{\pi(X_u) = 0} \d u \leq \sum_{k = 1}^{K(n^2 t) + 1} \int_{D^{k-1}}^{D^k} \indicator{\pi(X_u) = 0} \d u. \]
	
	The coupling with branching random walks shows that for each $k \geq 1$ we can write $D^k - D^{k-1} = E^k + V^k$, where $E^k$ and $V^k$ are independent and, under $\P$:
	\begin{itemize}
		\item $E^k = \int_{D^{k-1}}^{D^k} \indicator{\pi(X_u) = 0} \d u$ is the time that the price process stays at $0$ before the $k$th excursion starts. In particular, it follows an exponential distribution with parameter $\lambda \P(J = 1)$;
		\item $V^k$ is the time taken to explore the ambient tree $\Tree^k$, distributed according to $\Tree_1$, corresponding to the $k$th excursion. According to the discussion at the end of Section~\ref{sec:coupling}, we can write $V^k = \Scal^k(\tau(\Tree^k))$ where $\Scal^k$ is a random walk independent from $\Tree^k$ and with step distribution an exponential random variable with parameter $2 \lambda$.
	\end{itemize}
	
	Note furthermore that, since $\pi \circ X$ under $\P$ is regenerative at $0$, the random variables $((E^k, \Scal^k, \Tree^k), k \geq 1)$ are i.i.d.. With this decomposition, we have
	\[ K(y) = \sum_{k \geq 1} \indicator{D^k \leq y} = \sum_{k \geq 1} \indicator{ E^1 + \cdots + E^k + V^1 + \cdots + V^k \leq y} \leq \bar K(y) \]
	with $\bar K(y) = \sum_{k \geq 1} \indicator{ V^1 + \cdots + V^k \leq y}$, and so
	\[ \E^n(L^{n,\pi}_t) \leq \frac{1}{n} \E \left( \sum_{k = 1}^{\bar K(n^2 t) + 1} E^k \right) = \frac{1}{\lambda \P(J = 1) n} \E \left( \bar K(n^2 t) + 1 \right), \]
	where the last inequality follows from the independence between $\bar K(n^2 t)$ and the $E^k$'s. By definition of $\bar K(n^2 t)$ we have
	\[ \E \left( \bar K(n^2 t) \right) = \sum_{k \geq 1} \P \left( V^1 + \cdots + V^k \leq n^2 t \right) \leq \sum_{k \geq 1} \P \left( \max_{1 \leq i \leq k} V^i \leq n^2 t \right) \]
	and since the $V^i$'s are i.i.d.\ with common distribution $\Scal^1(\tau(\Tree^1))$, we end up with
	\[ \E \left( \bar K(n^2 t) \right) \leq \sum_{k \geq 1} \left\{ \P \left( V^1 \leq n^2 t \right) \right\}^k \leq \frac{1}{1 - \P \left( V^1 \leq n^2 t \right)} \]
	with this last upper bound being equal to $1/\P \left( \Scal^1(\tau(\Tree^1)) \geq n^2 t \right)$. According to~\eqref{eq:estimate-ell} we have $\P(\tau(\Tree_1) \geq u) \geq C u^{-1/2}$ and so the lower bound in~\eqref{eq:bounds-exploration} implies that $\P(\Scal^1(\tau(\Tree^1)) \geq u)$ obeys to a similar lower bound, which completes the proof.
\end{proof}

\begin{lemma} [Proof of property~\ref{prop:zero-set}] \label{lemma:lebesgue-measure}
	We have $\Pbf(\forall t \geq 0: \ell_t = 0) = 1$.
\end{lemma}

\begin{proof}
	We will use the compensation formula for the Poisson point process of excursions away from $0$ of $\pi \circ X$ associated to the local time $\ell$, see, e.g., Corollary IV.$11$ in~\cite{Bertoin96:0}. More precisely, under $\P^n$ the first jump of the right-continuous inverse of $\ell$ occurs at rate $\lambda \P(J = 1) n^2$, which uniquely identifies the excursion measure of $\pi \circ X$ associated to $\ell$ as being equal to $\lambda \P(J = 1) n^2$ times the law of an excursion of $\pi \circ X$ away from $0$, see for instance Proposition~O.$2$ in~\cite{Bertoin96:0}. In particular, if $(\beta_s, s \geq 0)$ is the Poisson point process of excursions of $\pi \circ X$ away from $0$ associated to $\ell$, so that $\beta_s \in \Ecal \cup \{\partial\}$ for some cemetery state $\partial$, then for any $t \geq 0$ and any measurable function $F: \Ecal \cup \{\partial\} \to [0,\infty)$ with $F(\partial) = 0$ we have (recall the notation $\E_1 = \E_{\delta_1}$)
	\[ \E^n \left( \sum_{0 \leq s \leq t} F(\beta_s) \right) = \lambda \P(J = 1) \times n \E^n \left( \ell_t \right) \times n \E^n_1 \big[ F(\sigma_{D_0}) \big]. \]
	
	Since $n \ell_t = L^{n, \pi}_t$, the previous lemma thus gives
	\begin{equation} \label{eq:compensation-formula}
		\E^n \left( \sum_{0 \leq s \leq t} F(\beta_s) \right) \leq C t^{1/2} \times n \E^n_1 \big[ F(\sigma_{D_0}) \big], \ n \geq 1, t \geq 0.
	\end{equation}
	
	Let us now prove the result: actually, it is enough to prove that for any $\eta > 0$,
	\begin{equation} \label{eq:goal-control-local-time-0}
		\limsup_{n \to +\infty} \P^n \left( \int_0^t \indicator{\pi(X_u) \leq \varepsilon} \d u \geq \eta \right) \mathop{\longrightarrow}_{\varepsilon \to 0} 0.
	\end{equation}
	
	Indeed, if this holds, then using the convergence $\int \phi \d X_t \Rightarrow^n \int \phi \d X_t$ for continuous $\phi$ with a compact support, this would imply
	\[ \Pbf\left(\int_0^t \indicator{\pi(X_u) \leq \varepsilon} \d u \geq \eta\right) \mathop{\longrightarrow}_{\varepsilon \to 0} 0 \]
	which would yield $\Pbf(\ell_t = 0) = 1$ for each fixed $t \geq 0$, and thus $\Pbf(\forall t \geq 0: \ell_t = 0) = 1$ by continuity of $\ell$. So let us show~\eqref{eq:goal-control-local-time-0}: using Markov inequality, writing $\int_0^t \indicator{\pi(X_u) \leq \varepsilon} \d u = \ell_t + \int_0^t \indicator{0 < \pi(X_u) \leq \varepsilon} \d u$ and using $t \leq D_t$, we obtain
	\[ \P^n \left( \int_0^t \indicator{\pi(X_u) \leq \varepsilon} \d u \geq \eta \right) \leq \frac{1}{\eta} \E^n(\ell_t) + \frac{1}{\eta} \E^n \left( \int_0^{D_t} \indicator{0 < \pi(X_u) \leq \varepsilon} \d u \right). \]
	
	Since $\ell_t = L^{n,\pi}_t / n$, the first term of the above upper bound is upper bounded by $C t^{1/2} / (\eta n)$ by Lemma~\ref{lemma:control-mean-local-time}. As for the second term, using~\eqref{eq:compensation-formula} with $F(\epsilon) = \int_0^{\zeta(\epsilon)} \indicator{\epsilon_u \leq \varepsilon} \d u$ for $\epsilon \in \Ecal$, we obtain
	\begin{multline*}
		\E^n \left( \int_0^{D_t} \indicator{0 < \pi(X_u) \leq \varepsilon} \d u \right) \leq C t^{1/2} n \E^n_1 \left( \int_0^{D_0} \indicator{\pi(X_u) \leq \varepsilon} \d u \right)\\
		= \frac{C t^{1/2}}{n} \E_1 \left( \int_0^{D_0} \indicator{\pi(X_u) \leq \varepsilon n} \d u \right).
	\end{multline*}
	
	In terms of the exploration of the ambient tree (equal in distribution to $\Tree_1$), the integral under the expectation corresponds to the time spent when the largest label of a green node was $\leq \varepsilon n$. Since transitions occur at rate $2 \lambda$ independently from everything else, we therefore have
	\[ \E_1 \left( \int_0^{D_0} \indicator{\pi(X_u) \leq \varepsilon n} \d u \right) = \frac{1}{2 \lambda} \E \left( \sum_{k = 0}^{\tau(\Tree_1)-1} \indicator {\psi( \gamma(\Phi_k(\Tree_1)), \Tree_1) \leq \varepsilon n} \right). \]
	
	Further, we can write
	\[ \sum_{k = 0}^{\tau(\Tree_1)-1} \indicator {\psi( \gamma(\Phi_k(\Tree_1)), \Tree_1) \leq \varepsilon n} = \sum_{p=1}^{\varepsilon n} \sum_{v \in B(\Tree_1)} \indicator{\psi(v, \Tree_1) = p} \sum_{k \geq 0} \indicator{\gamma(\Phi_k(\Tree_1)) = v}. \]
	
	The sum $\sum_{k \geq 0} \indicator{\gamma(\Phi_k(\Tree_1)) = v}$ counts the number of times the node $v$ has been the price: this is actually equal to $1 + \Ccal(v)$, with $\Ccal(v)$ the number of children of $v$ in $\Tree_1$. The one accounts for the first time $v$ becomes the price, and the additional $\Ccal(v)$ accounts for the fact that each child of $v$ makes $v$ stay the price one more unit of time (either immediately, if the child has a smaller label, or later on if the child has a larger label). Thus
	\[ \sum_{k = 0}^{\tau(\Tree_1)-1} \indicator {\psi( \gamma(\Phi_k(\Tree_1)), \Tree_1) \leq \varepsilon n} = \sum_{p=1}^{\varepsilon n} \sum_{v \in B(\Tree_1)} \indicator{\psi(v, \Tree_1) = p} (1 + \Ccal(v)). \]
	
	Now this sum counts twice all nodes in $B(\Tree_1) \setminus \Kcal(\Tree_1)$ with label in $\{1, \ldots, \varepsilon n\}$; it also counts once the nodes in $\Kcal(\Tree_1)$ as well as the nodes with label $\varepsilon n + 1$ whose parent has label $\varepsilon n$. In particular,
	\begin{equation} \label{eq:factor-2}
		\sum_{k = 0}^{\tau(\Tree_1)-1} \indicator {\psi( \gamma(\Phi_k(\Tree_1)), \Tree_1) \leq \varepsilon n} \leq 2 \sum_{v \in B(\Tree_1)} \indicator{\psi(v, \Tree_1) \leq \varepsilon n + 1}
	\end{equation}
	and so taking the mean and using~\eqref{eq:estimate-lebesgue-measure} finally gives $\E_1(\int_0^{D_0} \indicator{\pi(X_u) \leq \varepsilon n} \d u) \leq C \varepsilon n$. Gathering the previous inequalities, we end up with
	\[ \P^n \left( \int_0^t \indicator{\pi(X_u) \leq \varepsilon} \d u \geq \eta \right) \leq \frac{C t^{1/2}}{\eta n} + \frac{C \varepsilon t^{1/2}}{\eta} \]
	from which~\eqref{eq:goal-control-local-time-0} follows by letting first $n \to +\infty$ and then $\varepsilon \to 0$.
\end{proof}

\begin{lemma} [Proof of properties~\ref{prop:D'} and~\ref{prop:reg}] \label{lemma:D}
	We have $\Pbf(D'_0 = 0) = 1$ and for any $t \geq 0$ and $f, g$ bounded, continuous functions on $\Ecal$, the relation~\eqref{eq:def-continuous-regeneration} holds.
\end{lemma}

\begin{proof}
	We first prove the property~\ref{prop:reg}. First, assume that for every $\eta > 0$ it holds that
	\begin{equation} \label{eq:D_t-1}
		\limsup_{n \to +\infty} \P^n \left( D_t - D_{t, \varepsilon} \geq \eta \right) \mathop{\longrightarrow}_{\varepsilon \to 0} 0.
	\end{equation}
	
	As explained in Section~\ref{sub:aside-hitting-times}, this implies that $(\theta_{D_t}, \sigma_{D_t}) \Rightarrow^n (\theta_{D_t}, \sigma_{D_t})$. On the other hand, since $\pi \circ X$ under $\P^n$ is regenerative at $0$,~\eqref{eq:def-continuous-regeneration} holds with $\E^n$ instead of $\Ebf$ and so passing to the limit and using $(\theta_{D_t}, \sigma_{D_t}) \Rightarrow^n (\theta_{D_t}, \sigma_{D_t})$ we obtain the desired result. Thus we only have to prove~\eqref{eq:D_t-1}, which we do now.
	
	Let $A_t = t - G_t$ be the age of the excursion straddling $t$, and $G^u < D^u$ for $u > 0$ be the endpoints of the first excursion of $\pi \circ X$ with length $> u$, say $e^u$: then Theorem~$(5.9)$ in~\cite{Getoor79:0} shows that for $u > 0$ and $\nu \in \Mcal_F$ and conditionally on $\{ A_t = u, X^n_{G_t} = \nu \}$ (recall the definition of $X^n_t$ before Theorem~\ref{thm:main}), the excursion of $\pi \circ X$ straddling $t$ is equal in distribution to $e^u$ conditionally on $\{ X^n_{G^u} = \nu \}$: in particular,
	\begin{multline*}
		\P^n \left( D_t - D_{t, \varepsilon} \geq \eta \right)\\
		= \int \P^n \left( A_t \in \d u, X^n_{G_t} \in \d \nu \right) \P^n_\nu \left( D_u - D_{u, \varepsilon} \geq \eta \mid D_0 > u \right).
	\end{multline*}
	
	Further, under $\P^n$, $X^n_{G_t}$ is almost surely of the form $\nu = \varsigma \delta_0 + \delta_1$ for some $\varsigma \geq 0$. For such an initial condition, the number $\varsigma$ of orders sitting at $0$ does not influence the first excursion, which is distributed like the first excursion under $\P_1$: thus
	\[ \P^n \left( D_t - D_{t, \varepsilon} \geq \eta \right) = \int \P^n \left( A_t \in \d u \right) \P^n_1 \left( D_u - D_{u, \varepsilon} \geq \eta \mid D_0 > u \right) \]
	and the goal is now to prove that
	\[ \limsup_{n \to +\infty} \ \sup_{u > 0} \ \P^n_1 \left( D_u - D_{u, \varepsilon} > \eta \mid D_0 > u \right) \mathop{\longrightarrow}_{\varepsilon \to 0} 0, \]
	which will achieve the proof of~\eqref{eq:D_t-1}. Rescaling, we obtain
	\[ \P^n_1 \left( D_u - D_{u, \varepsilon} \geq \eta \mid D_0 > u \right) = \P_1 \left( D_{u n^2} - D_{u n^2, \varepsilon n} \geq \eta n^2 \mid D_0 > u n^2 \right) \]
	and to control this term we consider any $\varepsilon' > 0$ and write
	\begin{multline} \label{eq:bound-D}
		\P^n_1 \left( D_u - D_{u, \varepsilon} \geq \eta \mid D_0 > u \right) \leq \P_1 \left( S \geq (\varepsilon + \varepsilon') n \mid D_0 > u n^2 \right)\\
		+ \P_1 \left( D_{u n^2} - D_{u n^2, \varepsilon n} \geq \eta n^2, S \leq (\varepsilon + \varepsilon') n \mid D_0 > u n^2 \right)
	\end{multline}
	where $S = \sup \pi \circ X$, where the supremum is taken over $[D_{u n^2, \varepsilon n}, D_{u n^2}]$.
	\\
	
	\noindent{\textit{High-level description.}} Let us explain in words how we are going to upper bound each term in the right-hand side of~\eqref{eq:bound-D}: this reasoning will also be used in the proof of Lemma~\ref{lemma:g-d}. Let $\hat \Tree$ be the ambient tree corresponding to the first excursion of $\pi \circ X$ away from $0$, so that $D_0$ is the sum of $\tau(\hat \Tree)$ i.i.d.\ exponential random variables with parameter $2 \lambda$ and the conditioning $D_0 > u n^2$ therefore amounts, by~\eqref{eq:formula-tau}, to $B(\hat \Tree)$ having a large number of nodes.
	
	When $S \leq (\varepsilon + \varepsilon') n$, then $D_{u n^2} - D_{u n^2, \varepsilon n}$ is smaller than the time spent exploring \emph{all} the nodes in $\hat \Tree$ with label $\leq (\varepsilon + \varepsilon') n$. We have a good control on the number of such nodes (they are of the order of $(\varepsilon + \varepsilon')^2 n^2$) which thus translates into a good control on $D_{u n^2} - D_{u n^2, \varepsilon n}$ in this event.
	
	On the other hand, to control the probability of $S$ being large, i.e., $S > (\varepsilon + \varepsilon') n$, we observe that $S$ is equal to the largest supremum of the excursions above $\varepsilon n$ that start between times $D_{u n^2, \varepsilon n}$ and $D_0$. By Lemma~\ref{lemma:discrete-reg} these excursions are i.i.d.\ with common distribution the exploration of a tree distributed like $\Tree_1$. In particular, we can control their supremum (which is equal in distribution to $\psi^*(\Tree_1)$), and to control their number, we use a crude upper bound by saying that there cannot be more excursions above level $\varepsilon n$ than there are nodes in $\hat \Tree$ with label $= \varepsilon n$. Again, we have a good control on these two quantities which, combined, will give us a sufficiently good control on the probability of $S$ being large.
	\\
	
	Let us now make these arguments rigorous. As just explained, $D_{u n^2} - D_{u n^2, \varepsilon n}$ is, in the event $\{S \leq (\varepsilon + \varepsilon') n, D_0 > u n^2\}$, smaller than the time spent exploring the nodes of the ambient tree that have a label $\leq (\varepsilon + \varepsilon') n$. This means that if
	\[ N^\leq = \sum_{v \in B(\Tree_1)} \indicator{\psi(v, \Tree_1) \leq (\varepsilon + \varepsilon') n} \]
	is the number of such nodes, then
	\begin{multline*}
		\P_1 \left( D_{u n^2} - D_{u n^2, \varepsilon n} \geq \eta n^2, S \leq (\varepsilon + \varepsilon') n \mid D_0 > u n^2 \right)\\
		\leq \P \left( \Scal(2 N^\leq) \geq \eta n^2 \mid \tau(\Tree_1) > u n^2 \right)
	\end{multline*}
	(the factor $2$ in $2 N^\leq$ comes from the same reason as the $2$ in the right-hand side of~\eqref{eq:factor-2}). Invoking~\eqref{eq:bounds-exploration}, we get
	\begin{multline*}
		\P_1 \left( D_{u n^2} - D_{u n^2, \varepsilon n} \geq \eta n^2, S \leq (\varepsilon + \varepsilon') n \mid D_0 > u n^2 \right)\\
		\leq \P \left( N^\leq \geq \lambda \eta n^2 / 2 \mid \tau(\Tree_1) > u n^2 \right) + e^{-\overline \mu \eta n^2}
	\end{multline*}
	and so~\eqref{eq:estimate-D-1} finally gives
	\[ \P_1 \left( D_{u n^2} - D_{u n^2, \varepsilon n} \geq \eta n^2, S \leq (\varepsilon + \varepsilon') n \mid D_0 > u n^2 \right) \leq \frac{C (\varepsilon + \varepsilon')^2}{\eta} + e^{-\overline \mu \eta n^2}. \]
	
	We now control the second term in the right-hand side of~\eqref{eq:bound-D}. Let $e^k$ be the $k$th excursion of $\pi \circ X$ above level $\varepsilon n-1$ to start after time $un^2$, and let $N^=$ be the number of excursions above level $\varepsilon n-1$ that belong to the first excursion of $\pi \circ X$ away from $0$: then as explained above, for any $\kappa_0 > 0$ we have
	\begin{multline*}
		\P_1 \left( S \geq (\varepsilon + \varepsilon') n \mid D_0 > u n^2 \right) \leq \P_1 \left( N^= \geq \kappa_0 n \mid D_0 > u n^2 \right)\\
		+ \P_1 \left( \sup_{1 \leq k \leq \kappa_0 n} \sup e^k \geq \varepsilon' n \mid D_0 > u n^2 \right).
	\end{multline*}
	
	Thanks to the coupling, we have
	\[ \P_1 \left( N^= \geq \kappa_0 n \mid D_0 > u n^2 \right) \leq \P \left( \sum_{v \in \Tree_1} \indicator{\psi(v, \Tree_1) = \varepsilon n} \geq \kappa_0 n \mid \tau(\Tree_1) > u n^2 \right) \]
	and so~\eqref{eq:estimate-D-2} gives $\P_1 \left( N^= \geq \kappa_0 n \mid D_0 > u n^2 \right) \leq C \varepsilon/\kappa_0$. On the other hand, since under $\P_1(\, \cdot \mid D_0 > u n^2)$ the $(e^k, k \geq 1)$ are i.i.d., with common distribution the first excursion of $\pi \circ X$ under $\P_1$ (as a consequence of Lemma~\ref{lemma:discrete-reg}), we have thanks to the union bound
	\[ \P_1 \left( \sup_{1 \leq k \leq \kappa_0 n} \sup e^k \geq \varepsilon' n \mid D_0 > u n^2 \right) \leq \kappa_0 n \P_1 \left( \sup_{[0, D_0]} \pi \circ X \geq \varepsilon'n \right). \]
	
	By the coupling,
	\[ \P_1 \left( \sup_{[0, D_0]} \pi \circ X \geq \varepsilon'n \right) = \P \left( \psi^*(B(\Tree_1)) \geq \varepsilon' n \right) \leq \frac{C}{\varepsilon' n} \]
	where the last inequality follows from Lemma~\ref{lemma:tail-psi}. Gathering the previous bounds, we see that
	\[ \P^n_1 \left( D_u - D_{u, \varepsilon} \geq \eta \mid D_0 > u \right) \leq \frac{C (\varepsilon + \varepsilon')^2}{\eta} + \frac{C \varepsilon}{\kappa_0} + \frac{C \kappa_0}{\varepsilon'} + e^{-\overline \mu \eta n^2}. \]

	Choosing $\kappa_0 = \varepsilon^{1/2}$ and $\varepsilon' = \varepsilon^{1/4}$, and letting first $n \to \infty$ and then $\varepsilon \to 0$ achieves to prove~\eqref{eq:D_t-1}, and in particular property~\ref{prop:reg}.
	\\
	
	We now prove property~\ref{prop:D'}, i.e., for any $\varepsilon > 0$ we must prove that $\Pbf(D'_0 \geq \varepsilon) = 0$. For any $\eta > 0$ we have $\Pbf(D'_0 \geq \varepsilon) \leq \Pbf(D_\eta \geq \varepsilon - \eta)$. Because we have just proved that $D_\eta \Rightarrow^n D_\eta$, we have $\P^n \left( D_\eta \geq \varepsilon' \right) \to \Pbf \left( D_\eta \geq \varepsilon' \right)$ for all $\varepsilon'$ outside a countable set. Adapting the previous arguments, it is on the other hand not difficult to see that
	\[ \limsup_{n \to +\infty} \P^n \left( D_\eta \geq \varepsilon' \right) \mathop{\longrightarrow}_{\eta \to 0} 0 \]
	which concludes the proof of the lemma.
\end{proof}

\mysubsection{a regenerative property at the excursion level} \label{sub:R}

Let in the sequel $\Ncal$ be an excursion measure of $\Pbf$, whose existence has been proved in the previous step. With a slight abuse in notation we will consider that $\Ncal$ acts on measurable functions $f: \Ecal \to [0,\infty)$ by $\Ncal(f) = \int f \d \Ncal$. Note that $\Ncal$ is only determined up to a multiplicative constant: in this step the value of this multiplicative constant is irrelevant (because we only consider $\Ncal$ upon some conditionings), and it will be fixed at the end of the next step.

The goal of this step is to show that $\Ncal$ satisfies the following regenerative property~\RP\ studied in~\cite{Weill07:0}. In the sequel we use the canonical notation for excursions, and let $\epsilon = (\epsilon_t, t \geq 0)$ denote the canonical excursion and $\xi(a, u)$ for $a, u > 0$ denote the number of excursions of $\epsilon$ above level $a$ that have height $> u$.

\begin{itemize}[align=left, leftmargin=*]
 \item[(R)] \label{reg-property} For every $a, u > 0$ and $p \in \N$, under the probability measure $\Ncal(\, \cdot \mid \sup \epsilon > a)$ and conditionally on the event $\{\xi(a, u) = p\}$, the $p$ excursions of $\epsilon$ above level $a$ with height greater than $u$ are independent and distributed according to the probability measure $\Ncal(\, \cdot \mid \sup \epsilon > u)$.
\end{itemize}

This property implies that $\Ncal$ is the law of the excursion height process of a spectrally positive L\'evy process that does not drift to $+\infty$, see the next step for more details.
\\

The rest of this step is therefore devoted to proving that $\Ncal$ satisfies the regenerative property~\RP. Fix until the rest of this step $a, u > 0$, $p \in \N$ and $(f_k, k = 1, \ldots, p)$ continuous, bounded and non-negative functions on $\Ecal$. Consider the first excursion of $\epsilon$ (or $\pi \circ X$) with exactly $p$ excursions above $a$ with height larger than $u$ and let $(\hat \epsilon^k, k = 1, \ldots, p)$ be these $p$ excursions: in order to show that $\Ncal$ satisfies~\RP\ we have to show that
\[ \Ncal \left( \prod_{k=1}^p f_k(\hat \epsilon^k) \mid \xi(a, u) = p \right) = \prod_{k=1}^p \Ncal \left( f_k(\epsilon) \mid \sup \epsilon > u \right). \]

To prove this we will prove that
\begin{equation} \label{eq:goal-RP-1}
	\E^n \left( \prod_{k=1}^p f_k(\hat \epsilon^k) \right) \mathop{\longrightarrow}_{n \to +\infty} \Ncal \left( \prod_{k=1}^p f_k(\hat \epsilon^k) \mid \xi(a, u) = p \right)
\end{equation}
while at the same time
\begin{equation} \label{eq:goal-RP-2}
	\E^n \left( \prod_{k=1}^p f_k(\hat \epsilon^k) \right) \mathop{\longrightarrow}_{n \to +\infty} \prod_{k=1}^p \Ncal \left(f_k(\epsilon) \mid \sup \epsilon > u \right).
\end{equation}

Let in the rest of the proof $\hat g^k < \hat d^k$ be the endpoints of $\hat \epsilon^k$, $\epsilon^k$ be the $k$th excursion of $\pi \circ X$ above $a$ with height $> u$, and $g^k < d^k$ be its endpoints. Note in particular that $(g^1, d^1) = (g_{a,a+u}, d_{a,a+u})$ (recall the definitions of $g_{a,b}$ and $d_{a,b}$ in~\eqref{eq:def-g-d}--~\eqref{eq:def-U}). Before delving into the technical details let us give an high-level overview of the proofs of~\eqref{eq:goal-RP-1} and~\eqref{eq:goal-RP-2}.
\\

\noindent{\textit{High-level overview of the proof of~\eqref{eq:goal-RP-1}.}} The first step in the proof of~\eqref{eq:goal-RP-1} is to reduce the proof to showing that the endpoints $(\hat g^k, \hat d^k)$ of the $\hat \epsilon^k$ converge. Next, the $(\hat g^k, \hat d^k)$ form by definition a subsequence of the $(g^k, d^k)$: more precisely there exists $k^*$ such that $(\hat g^k, \hat d^k) = (g^{k^*+k}, d^{k^*+k})$, which further reduces to proving that the $(g^k, d^k)$ converge jointly with $k^*$. Finally, we can express $k^*$ in terms of return times to $0$ which eventually reduces the whole proof to the convergence of suitably chosen random times. To do so we prove that the limiting process must cross the levels it reaches in the sense explained in Section~\ref{sub:aside-hitting-times}, i.e., we have to prove results similar to~\eqref{eq:hitting-times}. As before, such controls will be provided by the coupling with the branching random walk.
\\

\noindent{\textit{High-level overview of the proof of~\eqref{eq:goal-RP-2}.}} Heuristically,~\eqref{eq:goal-RP-2} means that the $\hat \epsilon^k$ are asymptotically independent. This is very reasonable as the only correlation between successive excursions above a given level $a$ is through the number of orders placed below $a$ during a given excursion, and this number is small because orders can only be placed below $a$ when the price is below $a+j^*$. In other words, the main intuition behind the proof is that the measures $X^n_{d^1}$ and $X^n_{g^1-}$ representing the state of the book before and after the first excursion above $a$ are very close. This is the meaning of Lemma~\ref{lemma:varepsilon} below, which is proved thanks to coupling argument.

\subsubsection{Proof of~\eqref{eq:goal-RP-1}}

Since $\Ncal$ is an excursion measure of $\pi \circ X$ under $\Pbf$, the probability distribution $\Ncal(\, \cdot \mid \xi(a,u) = p)$ is the law of the first excursion of $\pi \circ X$ under $\Pbf$ that has exactly $p$ excursions above $a$ with height $> u$, and in particular
\[ \Ebf \left( \prod_{k=1}^p f_k(\hat \epsilon^k) \right) = \Ncal \left( \prod_{k=1}^p f_k(\hat \epsilon^k) \mid \xi(a, u) = p \right). \]

Thus in order to prove~\eqref{eq:goal-RP-1} it is enough to show that
\[ \E^n \left( \prod_{k=1}^p f_k(\hat \epsilon^k) \right) \mathop{\longrightarrow}_{n \to +\infty} \Ebf \left( \prod_{k=1}^p f_k(\hat \epsilon^k) \right), \]
i.e., that $(\hat \epsilon^k, k = 1, \ldots, p) \Rightarrow^n (\hat \epsilon^k, k = 1, \ldots, p)$. In view of Lemma~\ref{lemma:continuity-shift+stopping} it is enough to show that the corresponding endpoints converge, i.e., that we have the convergence $((\hat g^k, \hat d^k), k = 1, \ldots, p) \Rightarrow^n ((\hat g^k, \hat d^k), k = 1, \ldots, p)$. We first show in the following two lemmas that $T_{a+u}$ and $(g_{a,a+u}, d_{a,a+u})$ converge, and explain after Lemma~\ref{lemma:g-d} why this implies the convergence of $((\hat g^k, \hat d^k), k = 1, \ldots, p)$. In the following lemma, the limit means that the left-hand side can be made arbitrarily small by choosing $b - \underline b, \overline b - b \geq 0$ small enough. The limit in Lemma~\ref{lemma:g-d} has the same meaning.

\begin{lemma}\label{lemma:T}
	For any $b, \eta > 0$,
	\[ \limsup_{n \to +\infty} \P^n \left( T_{\overline b} - T_{\underline b} \geq \eta \right) \mathop{\longrightarrow}_{\overline b \downarrow b, \underline b \uparrow b} 0. \]
\end{lemma}

\begin{proof}
	Consider any $0 < \underline b' < \underline b$ and let $\varepsilon = \overline b - \underline b$, $\varepsilon' = \underline b - \underline b'$ and $S = \sup \pi \circ X - \underline b' n$, where the supremum is taken over $[g_{\underline b' n, \underline b n}, d_{\underline b' n, \underline b n}]$. Rescaling, we obtain
	\begin{multline*}
		\P^n \left( T_{\overline b} - T_{\underline b} > \eta \right) = \P \left( T_{\overline b n} - T_{\underline b n} > \eta n^2, S \geq (\varepsilon + \varepsilon') n \right)\\
		+ \P \left( T_{\overline b n} - T_{\underline b n} > \eta n^2, S < (\varepsilon + \varepsilon') n \right).
	\end{multline*}

	When $S \geq (\varepsilon + \varepsilon') n$, then necessarily $g_{\underline b' n, \underline b n} \leq T_{\underline b n} \leq T_{\overline b n} \leq d_{\underline b' n, \underline b n}$ and so
	\[ \P^n \left( T_{\overline b} - T_{\underline b} > \eta \right) \leq \P \left( U_{\underline b' n, \underline b n} > \eta n^2 \right) + \P \left( S < (\varepsilon + \varepsilon') n \right). \]
	
	The coupling implies that $U_{\underline b' n, \underline b n}$ under $\P$ is equal in distribution to $\Scal(\tau(\Tree_1) + 1)$ conditionally on $\psi^*(B(\Tree_1)) \geq \varepsilon' n$, and also that $S$ under $\P$ is equal in distribution to $\psi^*(B(\Tree_1))$ conditioned on $\psi^*(B(\Tree_1)) \geq \varepsilon' n$. Using $\tau(\Tree_1) + 1 \leq 2 \lvert \Tree_1 \rvert$ by~\eqref{eq:formula-tau} and using also~\eqref{eq:bounds-exploration}, we therefore get
	\begin{multline*}
		\P^n \left( T_{\overline b} - T_{\underline b} > \eta \right) \leq \P \left( \lvert \Tree_1 \rvert > \lambda \eta n^2 / 3 \mid \psi^*(B(\Tree_1)) \geq \varepsilon' n \right) + e^{-\overline \mu \eta n^2}\\
		+ \P \left( \psi^*(B(\Tree_1)) < (\varepsilon + \varepsilon') n \mid \psi^*(B(\Tree_1)) \geq \varepsilon' n \right).
	\end{multline*}
	
	In view of~\eqref{eq:estimate-T-1} and~\eqref{eq:estimate-T-2}, choosing $\varepsilon' = \varepsilon^{1/2}$ and letting first $n \to +\infty$ and then $\varepsilon \to 0$ gives the result.
\end{proof}

Recall the definition of $U_{a,b}$ in~\eqref{eq:def-g-d}--\eqref{eq:def-U}.

\begin{lemma}\label{lemma:g-d}
	For any $0 \leq a < b$ and any $\eta > 0$,
	\begin{equation} \label{eq:g-d}
		\limsup_{n \to +\infty} \P^n \left( U_{\underline a, b} - U_{\overline a, b} \geq \eta \right) \mathop{\longrightarrow}_{\overline a \downarrow a, \underline a \uparrow a} 0.
	\end{equation}
\end{lemma}

\begin{proof}
	Since $U_{\underline a, b} - U_{\overline a, b} = (d_{\underline a, b} - d_{\overline a, b}) + (g_{\overline a, b} - g_{\underline a, b})$, we only have to prove that
	\[ \lim_{\overline a \downarrow a, \underline a \uparrow a} \limsup_{n \to +\infty} \P^n \left( d_{\underline a, b} - d_{\overline a, b} \geq \eta \right) = \lim_{\overline a \downarrow a, \underline a \uparrow a} \limsup_{n \to +\infty} \P^n \left( g_{\overline a, b} - g_{\underline a, b} \geq \eta \right) = 0. \]
	
	The proofs for $d$ and $g$ are very similar to one another, and also very similar to the proof of Lemma~\ref{lemma:D}. Let us first sketch the proof for $d$. First of all, we are interested in the excursion straddling $T_{b n}$ and above $\underline a n$, so the ambient tree, say $\hat \Tree$, is distributed like $\Tree_1$ conditioned on $\psi^*(B(\Tree_1)) > (b - \underline a) n$. Let $\varepsilon = \overline a - \underline a$, consider any $\varepsilon' > 0$ and define
	\[ N^\leq = \sum_{v \in \hat \Tree} \indicator{\psi(v, \hat \Tree) \leq (\varepsilon + \varepsilon') n} \ \text{ and } \ N^= = \sum_{v \in \hat \Tree} \indicator{\psi(v, \hat \Tree) = \varepsilon n} \]
	
	Rescaling and introducing $S = \sup \pi \circ X - \underline a n$, where the supremum is taken over $[d_{\overline a n, b n}, d_{\underline a n, b n}]$, we obtain
	\begin{multline} \label{eq:bound-d}
		\P^n \left( d_{\underline a, b} - d_{\overline a, b} \geq \eta \right) \leq \P \left( d_{\underline a n, b n} - d_{\overline a n, b n} \geq \eta n^2, S \leq (\varepsilon + \varepsilon') n \right)\\
		+ \P \left( S > (\varepsilon + \varepsilon') n \right).
	\end{multline}
	
	To control the right-hand side of the above upper bound we use a similar reasoning as in the proof of Lemma~\ref{lemma:D} (see the \emph{High-level description} there). To control the first term of the above right-hand side, we observe that in the event $\{S \leq (\varepsilon + \varepsilon') n\}$, $d_{\underline a n, b n} - d_{\overline a n, b n}$ is upper bounded by the time spent exploring nodes with label $\leq (\varepsilon + \varepsilon') n$ in $\hat \Tree$, which leads to the bound
	\begin{multline*}
		\P \left( d_{\underline a n, b n} - d_{\overline a n, b n} \geq \eta n^2, S \leq (\varepsilon + \varepsilon') n \right)\\
		\leq \P \left( N^\leq \geq \lambda \eta n^2 / 2 \mid \psi^*(B(\Tree_1)) > (b - \underline a) n \right) + e^{- \overline \mu \eta n^2}.
	\end{multline*}

	To control the second term of the right-hand side of~\eqref{eq:bound-d}, we observe that $(1)$ $S$ is equal to the largest supremum of the excursions above $\overline a n$ that start after $d_{\overline a n, b n}$ and end before $d_{0, b n}$; $(2)$ the number of such excursions is smaller than the number of nodes with label $= \varepsilon n$ in $\hat \Tree$; and $(3)$ the excursions above $\overline a n$ and starting after time $d_{\overline a n, b n}$ are i.i.d.\ with common distribution the first excursion of $\pi \circ X$ under $\P^n_1$. This leads to the bound
	\[ \P \left( S > (\varepsilon + \varepsilon') n \right) \leq \P \left( N^= \geq \kappa_0 n \mid \psi^*(B(\Tree_1)) > (b - \underline a) n \right) + \frac{C \kappa_0}{\varepsilon'}. \]
	
	In view of~\eqref{eq:estimate-g-d-1} and~\eqref{eq:estimate-g-d-2} we get the desired result for $d$. For $g$ we derive the exact same upper bound by considering $S = \sup \pi \circ X - \underline a n$, where the supremum is now taken over $[g_{\underline a n, b n}, g_{\overline a n, b n}]$. There is one minor difference, namely that excursions above $\overline a n$ that end before $g_{\overline a n, b n}$ are i.i.d., but with distribution the first excursion of $\pi \circ X$ above $\overline a n$ conditioned on having its height $< b n$. Since the probability of this event goes to one, this additional conditioning does not influence the final result.
\end{proof}

We now explain why the two previous lemmas imply the convergence of the vector $((\hat g^k, \hat d^k), k = 1, \ldots, p)$ (by which we mean that $((\hat g^k, \hat d^k), k = 1, \ldots, p) \Rightarrow^n ((\hat g^k, \hat d^k), k = 1, \ldots, p)$). First of all, the discussion in Section~\ref{sub:aside-hitting-times} shows that $\pi \circ X$ shifted at time $T_{a+u}$ converges. Thus by Lemma~\ref{lemma:g-d}, $d_{a, a + u}$, which is the hitting time of $(0, a]$ by the process $\pi \circ X$ shifted at time $T_{a+u}$, converges. Moreover, the arguments in the proof of Lemma~\ref{lemma:g-d} go through for $a = 0$, which shows that $D_{d_{a,a+u}}$ converges. Since $g_{a, a + u}$ is the hitting time of $(0,a]$ by the process $\pi \circ X$ shifted at time $T_{a+u}$ and run backward in time, and the mapping that to a function associates the same function run backward in time is continuous, we obtain for the same reasons the convergence of $g_{a, a+u}$.

Recall that $\epsilon^k$ is the $k$th excursion of $\pi \circ X$ above $a$ with height $> u$, and $g^k < d^k$ are its endpoints. Let also $T^k = \inf\{ t \geq g^k: \pi(X_t) \geq a+u \}$. The idea is now to iterate the above arguments by looking at the process $\pi \circ X$ shifted at time $d^k$. Let us look at $k = 1$, for which we have $(g^1, d^1, T^1) = (g_{a,a+u}, d_{a,a+u}, T_{a+u})$. Inspecting the proof of Lemma~\ref{lemma:T}, we see that $T^2$ converges: indeed, all that matters in the proof of Lemma~\ref{lemma:T} is the local behavior around $b$, for which the initial state $\nu$, as long as $\pi(\nu)$ is far below $b$, is irrelevant (note that this is the case when shifting $\pi \circ X$ at time $d^1$, since by definition $\pi(X_{d^1}) \leq a$).

Moreover, since the successive excursions above $a$ are i.i.d.\ by Lemma~\ref{lemma:discrete-reg}, Lemma~\ref{lemma:g-d} implies, since $T^2$ converges, that $d^2$, $g^2$ and $D_{d^2}$ converge. Iterating, we obtain the convergence of $d^k$, $g^k$ and $D_{d^k}$ for every $k \geq 1$. Finally, it is not hard to see that these convergences hold jointly, i.e., $((g^k, d^k, D_{d^k}), k \geq 1) \Rightarrow^n ((g^k, d^k, D_{d^k}), k \geq 1)$. There are two different ways to see this: either use arguments as in end of the discussion in Section~\ref{sub:aside-hitting-times}, or use the fact that the results of Lemmas~\ref{lemma:T} and~\ref{lemma:g-d} actually show more than just weak convergence, but actually that the limiting functions have no fixed point of discontinuity, and then use the continuous mapping theorem.

Having the joint convergence with the $D_{d^k}$'s makes it possible to know whether two successive excursions above $a$ with height $> u$ belong to the same excursion away from $0$. In particular, if $k^* \geq 0$ is the first index such that $D_{d^{k^*}} < D_{d^{k^*+1}} = \cdots = D_{d^{k^*+p}} < D_{d^{k^*+p+1}}$ (defining $d^0 = 0$), then $(\hat g^k, \hat d^k) = (g^{k^* + k}, d^{k^* + k})$ for $k = 1, \ldots, p$. From the convergence of $((g^k, d^k, D_{d^k}), k \geq 1)$ we obtain the convergence of $k^*$, which therefore entails the convergence of $((\hat g^k, \hat d^k), k = 1, \ldots, p)$ as desired. This finally achieves the proof of~\eqref{eq:goal-RP-1}.

\subsubsection{Proof of~\eqref{eq:goal-RP-2}} Let $m^k = \sup \pi \circ X$, where the supremum is taken over the interval $[d^k, D_{d^k}]$: then $(\hat \epsilon^k, k = 1, \ldots, p)$ is equal in distribution to $(\epsilon^k, k = 1, \ldots, p)$ conditionally on $\{ m^p < a + u < m^1, \ldots, m^{p-1} \}$ (which is to be understood as $\{m^1 < a+u \}$ when $p = 1$). In particular,
\[ \E^n \left( \prod_{k=1}^p f_k(\hat \epsilon^k) \right) = \frac{\E^n \left[ \prod_{k=1}^p f_k(\epsilon^k) ; m^p < a + u < m^1, \ldots, m^{p-1} \right]}{\P^n \left( m^p < a + u < m^1, \ldots, m^{p-1} \right)}. \]

For $\nu \in \Mcal_F$ we define
\[ A^n_q(\nu) = \E^n_\nu \left( \prod_{k=1}^q f_{k+p-q}(\epsilon^k) ; m^q < a+u < m^1, \ldots, m^{q-1} \right) \ \text{ for } \ q = 1, \ldots, p, \]
so that
\[ \E^n \left( \prod_{k=1}^p f_k(\hat \epsilon^k) \right) = \frac{A^n_p(\zero)}{A^{n,1}_p(\zero)} \]
with $A^{n,1}_p$ defined similarly as $A^n_p$ by taking all the $f_k$'s equal to the constant function which takes value one. Moreover, let us introduce for $\nu \in \Mcal_F$
\[ \left \{ \begin{array}{l}
	\displaystyle B^n_q(\nu) = \E^n_\nu \left( \prod_{k=1}^q f_{k+p-q}(\epsilon^k) ; g^1 < D_0, m^q < a+u < m^1, \ldots, m^{q-1} \right)\\
	\hspace{90mm}\text{for } q = 1, \ldots, p-1,\\
	\displaystyle B^n_0(\nu) = \P^n_\nu \left( D_0 < d^1 \right)\\
	\hspace{90mm}\text{for } q = 0,\\
	\displaystyle \Delta^n_q(\nu) = \E^n_\nu \left[ f_{p-q}(\epsilon^1) \times \left( B^n_q(X^n_{d^1}) - B^n_q(X^n_{g^1-})\right) \right]\\
	\hspace{90mm}\text{for } q = 0, \ldots, p-1
\end{array} \right. \]
(recall that $\nu^n \in \Mcal_F$ is the measure such that $\vartheta_n(\nu^n) = \nu$) and finally
\[ \varphi^n_q = \E^n \left[ f_{p-q+1}(\epsilon^1) \right] \ \text{ for } \ q = 1, \ldots, p. \]

Note that Lemma~\ref{lemma:discrete-reg} together with Lemmas~\ref{lemma:continuity-shift+stopping},~\ref{lemma:T},~\ref{lemma:g-d} and the definition of $\Ncal$ imply that
\begin{equation} \label{eq:conv-varphi}
	\varphi^n_q \mathop{\longrightarrow}_{n \to +\infty} \Ncal\left(f_{p-q+1}(\epsilon) \mid \sup \epsilon > u \right) \ \text{ for } \ q = 1, \ldots, p.
\end{equation}

We now derive some relations between all these quantities. First of all,
\begin{equation} \label{eq:B-A}
	B^n_q(\nu) = A^n_q(\nu) - B^n_0(\nu) A^n_q(\zero) \ \text{ for } \ q = 1, \ldots, p-1.
\end{equation}

Indeed, for $q = 1, \ldots, p-1$ we have
\begin{multline*}
	B^n_q(\nu) = A^n_q(\nu)\\
	- \E^n_\nu \left( \prod_{k=1}^q f_{k+p-q}(\epsilon^k) ; D_0 < g^1, m^q < a+u < m^1, \ldots, m^{q-1} \right),
\end{multline*}
and since $\pi \circ X$ regenerates at $D_0$, the second term of the above right-hand side is equal to
\[ \P^n_\nu(D_0 < g^1) \E^n \left( \prod_{k=1}^q f_{k+p-q}(\epsilon^k) ; m^q < a+u < m^1, \ldots, m^{q-1} \right). \]

Since the two events $\{D_0 < g^1\}$ and $\{D_0 < d^1\}$ coincide, we obtain~\eqref{eq:B-A}. Second, for $\nu \in \Mcal_F$ with $\pi(\nu) < an$, we have
\begin{equation} \label{eq:A-B}
	A^n_q(\nu) = \varphi^n_q \times \E^n_\nu \left[ B^n_{q-1}(X^n_{g^1-}) \right] + \Delta^n_{q-1}(\nu) \ \text{ for } \ q = 1, \ldots, p.
\end{equation}
Indeed, the strong Markov property at time $d^1$ gives for $q = 1, \ldots, p$
\begin{align*}
	A^n_q(\nu) & = \E^n_\nu \left[ f_{p-q+1}(\epsilon^1) \times B^n_{q-1}(X^n_{d^1}) \right]\\
	& = \E^n_\nu \left[ f_{p-q+1}(\epsilon^1) \times B^n_{q-1}(X^n_{g^1-}) \right] + \Delta^n_{q-1}(\nu).
\end{align*}
Since $\pi(\nu) < an$, under $\P^n_\nu$, $\epsilon^1$ and $X_{g^1-}$ are independent, and $\epsilon^1$ is distributed according to $\epsilon^1$ under $\P^n$, which gives~\eqref{eq:A-B}. Combining~\eqref{eq:B-A} and~\eqref{eq:A-B}, we end up with the following recursion for $A^n_q$:
\begin{multline} \label{eq:recursion}
	A^n_q(\nu) = \varphi^n_q \times \left[ \E^n_\nu \left( A^n_{q-1}(X^n_{g^1-}) \right) - A^n_{q-1}(\zero) \E^n_\nu \left( B^n_0(X^n_{g^1-}) \right) \right]\\
	+ \Delta^n_{q-1}(\nu)
\end{multline}
for $q = 2, \ldots, p$, and with the boundary condition
\begin{equation} \label{eq:boundary-condition}
	A^n_1(\nu) = \varphi^n_1 \times \E^n_\nu \left[ B^n_0(X^n_{g^1-}) \right] + \Delta^n_0(\nu).
\end{equation}

Since the functions $f_k$ were arbitrary in deriving this recursion, we obtain a similar recursion for $A^{n,1}_q(\nu)$, but with all the terms $\varphi^n_q$ replaced by one and the $\Delta^n_q(\nu)$'s replaced by $\Delta^{n,1}_q(\nu)$, defined similarly as $\Delta^n_q(\nu)$ but with the functions $f_q$ equal to the constant function taking value one. Now consider $\tilde A^n_q$ and $\tilde A^{n,1}_q$ that satisfy the same recursion~\eqref{eq:recursion}--\eqref{eq:boundary-condition}, but with all the $\Delta^n_q$ equal to $0$, i.e., for every $\nu \in \Mcal_F$,
\[ 	\tilde A^n_q(\nu) = \varphi^n_q \times \left[ \E^n_\nu \left( \tilde A^n_{q-1}(X^n_{g^1-}) \right) - \tilde A^n_{q-1}(\zero) \E^n_\nu \left( B^n_0(X^n_{g^1-}) \right) \right], \ q = 2, \ldots, p, \]
with the boundary condition
\[ \tilde A^n_1(\nu) = \varphi^n_1 \times \E^n_\nu \left[ B^n_0(X^n_{g^1-}) \right], \]
and similarly for $\tilde A^{n,1}_q(\nu)$ with all the terms $\varphi^n_q$ replaced by one. By induction one gets
\[ \frac{\tilde A^n_p(\zero)}{\tilde A^{n,1}_p(\zero)} = \prod_{k = 1}^p \varphi^n_k \]
and so~\eqref{eq:conv-varphi} implies that
\begin{equation} \label{eq:conv-tilde}
	\frac{\tilde A^n_p(\zero)}{\tilde A^{n,1}_p(\zero)} \mathop{\longrightarrow}_{n \to +\infty} \prod_{k=1}^p \Ncal \left( f_k(\epsilon) \mid \sup \epsilon > u \right).
\end{equation}

On the other hand, for $\nu$ with $\pi(\nu) < a n$, we have $A^n_1(\nu) - \tilde A^n_1(\nu) = \Delta^n_0(\nu)$ while for $q = 2, \ldots, p$,
\begin{multline*}
	\big \lvert A^n_q(\nu) - \tilde A^n_q(\nu) \big \rvert \leq \varphi^n_q \E^n_\nu \left( \big \lvert A^n_{q-1}(X^n_{g^1-}) - \tilde A^n_{q-1}(X^n_{g^1-}) \big \rvert \right)\\
	+ \varphi^n_q \E^n_\nu \left( B^n_0(X^n_{g^1-}) \right) \big \vert A^n_{q-1}(\zero) - \tilde A^n_{q-1}(\zero) \big \rvert + \big \lvert \Delta^n_{q-1}(\nu) \big \rvert.
\end{multline*}

Thus if $\varepsilon^{(n)} = \max_{q = 0, \ldots, p-1} \varepsilon^{(n)}_q$ with
\[ \varepsilon^{(n)}_q = \sup_{\nu \in \Mcal_F: \pi(\nu) < a n} \big \lvert \Delta^n_q(\nu) \big \rvert \ \text{ for } \ q = 0, \ldots, p-1, \]
then by induction we obtain $\lvert A^n_p(\zero) - \tilde A^n_p(\zero) \big \rvert \leq C \varepsilon^{(n)}$ for some finite constant $C$, and a similar upper bound holds for $\lvert A^{n,1}_p(\zero) - \tilde A^{n,1}_p(\zero) \rvert$ (note that, to perform the induction, we use the fact that $\pi(X^n_{g^1-}) < an$ $\P^n_\nu$-almost surely, for any $\nu$ with $\pi(\nu) < an$). In view of~\eqref{eq:conv-tilde}, the following lemma therefore achieves the proof of~\eqref{eq:goal-RP-2}.

\begin{lemma} \label{lemma:varepsilon}
	$\varepsilon^{(n)} \to 0$ as $n \to +\infty$.
\end{lemma}

\begin{proof}
	By definition we have for $q = 0, \ldots, p-1$
	\[ \lvert \Delta^n_q(\nu) \rvert \leq C \, \E^n_\nu \left[ \left\lvert B^n_q(X^n_{g^1-}) - B^n_q(X^n_{d^1}) \right\rvert \right] \]
	with $C = \max_{q=1,\ldots,p} \sup f_q$. To control the difference appearing in this last expectation, we use the following observation: under $\P^n_\nu$ for any $\nu \in \Mcal_F$ with $\pi(\nu) < an$, we can write $X^n_{d^1} = X^n_{g^1-} + \Xi$, where $\Xi \in \Mcal_F$ corresponds to the orders added below $a$ during the first excursion of $\pi \circ X$ above $a$ with height $> u$. In particular, the coupling implies that $\Xi$ is independent from $X^n_{g^1-}$, its law does not depend on $\nu$ and $M(\Xi)$ is equal in distribution to $\lvert \Kcal(\Tree_1) \rvert$ conditioned on $\psi^*(B(\Tree_1)) > u n$. In particular,
	\begin{equation} \label{eq:bound-varepsilon}
		\varepsilon^{(n)}_q \leq C \, \E^n \left[ \sup_{\nu \in \Mcal_F: \pi(\nu) < a n} \left\lvert B^n_q(\nu) - B^n_q(\nu + \Xi) \right\rvert \right].
	\end{equation}
	
	Thus we need to control terms of the form $B^n_q(\nu) - B^n_q(\nu + \tilde \nu)$ uniformly in $\nu \in \Mcal_F$ with $\pi(\nu) < an$, where $\tilde \nu$ plays the role of $\Xi$. In view of the definition of $B^n_q$, we thus need to understand the difference between $X$ under $\P^n_\nu$ and $\P^n_{\nu + \tilde \nu}$. More precisely, all the events and random variables involved in the computation of $B^n_q$ depend on $X$ stopped at $D_0$, and so we actually only need to compare the processes $(X_t, 0 \leq t \leq D_0)$ under $\P^n_\nu$ and $\P^n_{\nu + \tilde \nu}$.
	
	In order to do so we extend the coupling of Theorem~\ref{thm:coupling}: recall that this coupling couples $X$ under $\P_a$ with $\Tree_a$. Using this coupling, it is straightforward to couple $X$ under $\P_\nu$ with a \emph{forest} of trees $\Fcal(\nu) = (\Tcal_{(k)}, k = 1, \ldots, M(\nu))$ such that the trees $\Tcal_{(k)}$ are independent, and if $\nu = \sum_a \varsigma_a \delta_a$, then exactly $\varsigma_a$ of the trees $\Tcal_{(k)}$ are distributed like $\Tcal_a$. This coupling relies on extending the map $\Phi$ to make it act on forests in an obvious manner.
	
	This coupling between $\P_\nu$ and $\Fcal(\nu)$ provides a coupling between $\P_\nu$ and $\P_{\nu + \tilde \nu}$ as follows: first, one considers the forest $\Fcal(\nu)$ used to construct $X$ under $\P_\nu$. Then, one adds $M(\tilde \nu)$ independent trees to this forest, say $(\tilde \Tree_{(k)}, k = 1, \ldots, M(\tilde \nu))$, such that if $\tilde \nu = \sum_p \tilde \varsigma_p \delta_p$ then exactly $\tilde \varsigma_a$ of these trees are distributed according to $\Tcal_a$. We thus get a larger forest, say $\tilde \Fcal = \Fcal(\nu) \cup \{ \tilde \Tree_{(k)}, k = 1, \ldots, M(\tilde \nu) \}$, and by exploring this forest with successive iterations of $\Phi$ we get a new process $\tilde X$ on the same probability space that $X$. By construction and thanks to Theorem~\ref{thm:coupling}, this process is a version of $X$ under $\P_{\nu + \tilde \nu}$.
	
	Note moreover that, as mentioned previously, we are only interested in $X$ before time $D_0$. In particular, we can truncate the trees $\Tree_{(k)}$ and $\tilde \Tree_{(k)}$ by removing all the nodes that have a label $\leq 0$. It is thus convenient to consider the operator $B_0: \T \to \T$ that removes all the nodes of a tree $\tree \in \T$ with label $\leq 0$, as well as their descendants.
	
	If $\tilde \epsilon^k$, $\tilde g^1$, $\tilde m^k$ and $\tilde D_0$ are the quantities associated to $\tilde X$ in the same way that $\epsilon^k$, $g^1$, $m^k$ and $D_0$ are associated to $X$, then using the definition of $B^n_q$ we have
	\[ B^n_0(\nu) - B^n_0(\nu + \tilde \nu) = \E^n_\nu \left( \indicator{D_0 < d^1} - \indicator{\tilde D_0 < \tilde d^1} \right) \]
	for $q = 0$, while for $q = 1, \ldots, p-1$, $B^n_q(\nu) - B^n_q(\nu + \tilde \nu)$ is equal to
	\begin{multline*}
		\E^n_\nu \Bigg( \prod_{k=1}^q f_k(\epsilon^k) \indicator{g^1 < D_0, m^q < a+u < m^1, \ldots, m^{q-1}}\\
		- \prod_{k=1}^q f_k(\tilde \epsilon^k) \indicator{\tilde g^1 < \tilde D_0, \tilde m^q < a+u < \tilde m^1, \ldots, \tilde m^{q-1}} \Bigg).
	\end{multline*}
	Now the key observation is that in the event $\{ \max_k \psi^*(B_0(\tilde \Tree_{(k)})) < (a+u) n \}$, the two random variables (the one defined in terms of $X$ and the one defined in terms of $\tilde X$) in the previous expectation are equal. Indeed, in this event, the excursions above $a$ with height $> u$ for $X$ and $\tilde X$ coincide. In particular, since the random variables under consideration are bounded, we obtain
	\[ \lvert B^n_q(\nu) - B^n_q(\nu + \tilde \nu) \rvert \leq C \P^n_\nu \left( \max_{k = 1, \ldots, M(\tilde \nu)} \psi^*(B_0(\tilde \Tree_{(k)})) > (a+u) n \right). \]
	
	Recall that the trees $\tilde \Tree_{(k)}$ are independent. Further, $\psi^*(B_0(\Tree_y)) \leq \psi^*(\Tree_y)$, and $\psi^*(\Tree_y)$ is (stochastically) increasing in $y$, so that using the union bounds we get for any $\tilde \nu \in \Mcal_F$ with $\pi(\tilde \nu) < an$
	\[ \P^n_\nu \left( \max_{k = 1, \ldots, M(\tilde \nu)} \psi^*(B_0(\tilde \Tree_{(k)})) > (a+u) n \right) \leq M(\tilde \nu) \P \left( \psi^*(\Tree_0) > u n \right). \]

	In view of~\eqref{eq:bound-varepsilon} and the discussion preceding it, we therefore get
	\[ \varepsilon^{(n)}_q \leq C \E \left( \lvert \Kcal(\Tree_1) \mid \psi^*(B(\Tree_1)) > un \right) \times \P \left( \psi^*(\Tree_0) > u n \right). \]
	
	The supremum over $n \geq 1$ of the expectation in the above right-hand side is finite by~\eqref{eq:kcal}, and since $\P(\psi^*(\Tree_0) > un) \to 0$ as $n \to +\infty$, the result is proved.
\end{proof}

\mysubsection{$\pi \circ X$ under $\Pbf$ is a reflected Brownian motion} \label{step:price=RBM}

At this point, we know that $\Ncal$ is a $\sigma$-finite measure on $\Ecal$ that satisfies the following properties:
\begin{enumerate}[label=\Roman*)]
	\item $\Ncal(\zeta = + \infty) = 0$ (since $\pi \circ X$ under $\Pbf$ is upper bounded by $M \circ X$ by Lemma~\ref{lemma:joint-convergence}, which by Lemma~\ref{lemma:preliminary} is a Brownian motion with no drift reflected at $0$);
	\item $\Ncal(\epsilon \text{ is not continuous}) = 0$ (since $\pi \circ X$ under $\Pbf$ is almost surely continuous).
\end{enumerate}

In particular, $\Ncal$ induces a $\sigma$-finite measure $\Theta$ on the set of compact real trees via the usual coding of a compact real tree by a continuous excursion with finite length, see for instance~\cite[Section~$3$]{Le-Gall12:0}.
\\

Further, let $y > 0$ and $\epsilon^1$ be the first excursion of $\pi \circ X$ away from $0$ with height $> y$. Then by Lemmas~\ref{lemma:T} and~\ref{lemma:g-d}, we have $\P^n(\sup \epsilon^1 > x) \to \Pbf(\sup \epsilon^1 > x)$ for all $x$ outside a countable set, where this latter quantity is equal to $\Ncal(\sup \epsilon > x \mid \sup \epsilon > y)$ by definition of $\Ncal$. On the other hand,
\[ \P^n \left( \sup \epsilon^1 > x \right) = \P \left( \psi^*(B(\Tree_1)) > x n \mid \psi^*(B(\Tree_1)) > y n \right) \]
which, for any $0 < y < x$, converges to $y/x$ by Lemma~\ref{lemma:tail-psi}. Thus for all $x > y$ outside a countable set, we have
\[ \Ncal(\sup \epsilon > x \mid \sup \epsilon > y) = \frac{\Ncal(\sup \epsilon > x)}{\Ncal(\sup \epsilon > y)} = \frac{y}{x} \]
from which one deduces that $\Ncal(\sup \epsilon > x) = c / x$ for every $x > 0$, and for some finite constant $c > 0$ (this constant will be fixed shortly). Thus $\Ncal$ satisfies the following additional properties:
\begin{enumerate}[label=\Roman*)] \setcounter{enumi}{2}
	\item \label{prop:N-1} $\Ncal(\sup \epsilon = 0) = 0$ (by definition of an excursion measure);
	\item \label{prop:N-2} $0 < \Ncal(\sup \epsilon > x) < +\infty$ for every $x > 0$ (since $\Ncal(\sup \epsilon > x) = c/x$);
	\item \label{prop:N-3} $\Ncal(\Ecal) = +\infty$ (obtained by letting $x \downarrow 0$ in $\Ncal(\sup \epsilon > x) = c/x$);
	\item \label{prop:N-4} $\Ncal$ satisfies the regenerative property~\RP~(by the previous step).
\end{enumerate}

Properties~\ref{prop:N-1}--\ref{prop:N-3} above immediately translate to $\Theta$ having infinite mass, $\Theta(\Hcal = 0) = 0$ and $\Theta(\Hcal > x) \in (0,\infty)$ (where $\Hcal$ denotes the height of the canonical tree $\stree$). Moreover, the last property~\ref{prop:N-4} means exactly that $\Theta$ satisfies the property~\RP~of~\cite{Weill07:0}: indeed, excursions of $\epsilon$ above level $a$ under $\Ncal$ correspond to the subtrees of $\stree$ above level $a$ under $\Theta$. Finally, we see that the assumptions of Theorem~$1.1$ in~\cite{Weill07:0} are satisfied, which gives the existence of a spectrally positive L\'evy process $Y$, with Laplace exponent $\Psi$ satisfying $\int^\infty (1/\Psi) < +\infty$, such that $\Theta$ is the (excursion) law of the $\Psi$-L\'evy tree. In particular, $\Ncal$ is an excursion measure of the height process associated to $Y$.

We now fix the normalization constant as in~\cite{Duquesne02:0} (which amounts to choosing the constant $c$ above), so that according to Corollary~$1.4.2$ in~\cite{Duquesne02:0} (remember that $\int^\infty (1/\Psi) < +\infty$) we have
\[ \int_{\Ncal(\sup \epsilon > x)}^\infty \frac{\d u}{\Psi(u)} = x, \ x > 0, \]
which implies, since $\Ncal(\sup \epsilon > x) = c/x$, that $\Psi(u) = u^2 / c$. In other words, $Y$ is equal in distribution to $(2/c)^{1/2} \tilde W$, with $\tilde W$ a standard Brownian motion, and the height process associated to this L\'evy process is equal in distribution to $(2 c)^{1/2} W$ (to see this, consider for instance the CSBP $Z$ associated to $Y$, which has branching mechanism $\Psi$ and satisfies the SDE $\d Z_t = (2 Z_t / c)^{1/2} \d \tilde W_t$, and use~$(20)$ and~$(21)$ in Pardoux and~\cite{Pardoux11:0}). Since $\pi \circ X$ under $\Pbf$ is equal in distribution to the height process of $Y$, we obtain that $\pi \circ X$ under $\Pbf$ is equal in distribution to $(2c)^{1/2} W$. The following lemma makes it possible to identify $c$ and, more importantly, to conclude the proof of Theorem~\ref{thm:main}.

\begin{lemma} \label{lemma:asymptotic-local-time}
	For any $\eta > 0$,
	\[ \limsup_{n \to +\infty} \P^n \left( \left\lvert L^{n, \pi}_t - \frac{1}{\varepsilon} \int_0^t \indicator{\pi(X_u) \leq \varepsilon} \d u \right\rvert \geq \eta \right) \mathop{\longrightarrow}_{\varepsilon \downarrow 0} 0. \]
\end{lemma}

\begin{proof}
	Writing $\int_0^t \indicator{\pi(X_u) \leq \varepsilon} \d u = \ell_t + \int_0^t \indicator{0 < \pi(X_u) \leq \varepsilon} \d u$ and using the triangular inequality, we first obtain
	\begin{multline*}
		\P^n \left( \left\lvert L^{n, \pi}_t - \frac{1}{\varepsilon} \int_0^t \indicator{\pi(X_u) \leq \varepsilon} \d u \right\rvert \geq \eta \right) = \P^n \left( L^{n,\pi}_t \geq \eta \varepsilon n / 2 \right)\\
		+ \P^n \left( \left\lvert L^{n, \pi}_t - \frac{1}{\varepsilon} \int_0^t \indicator{0 < \pi(X_u) \leq \varepsilon} \d u \right\rvert \geq \eta / 2 \right).
	\end{multline*}
	
	The first term of the above upper bound goes to $0$ by Lemma~\ref{lemma:control-mean-local-time}, and so we need to control the second term. Rescaling leads to
	\begin{multline*}
		\P^n \left( \left\lvert L^{n, \pi}_t - \frac{1}{\varepsilon} \int_0^t \indicator{0 < \pi(X_u) \leq \varepsilon} \d u \right\rvert \geq \eta / 2 \right)\\
		= \P \left( \left\lvert \int_0^{n^2t} \indicator{\pi(X_u) = 0} \d u - \frac{1}{\varepsilon n} \int_0^{n^2t} \indicator{0 < \pi(X_u) \leq \varepsilon n} \d u \right\rvert \geq \eta n / 2 \right).
	\end{multline*}
	
	Let as in the proof of Lemma~\ref{lemma:control-mean-local-time} $K(y)$ be the number of excursions of $\pi \circ X$ away from $0$ that end before time $y$, $E^k$ be the time that $\pi \circ X$ stays at $0$ before the $k$th excursion and $V^k(y)$ be the time spent exploring nodes with label $\leq y$ in the $k$th ambient tree: then if $\pi(X_{n^2 t}) = 0$, we have
	\begin{multline*}
		\left\lvert \int_0^{n^2t} \indicator{\pi(X_u) = 0} \d u - \frac{1}{\varepsilon n} \int_0^{n^2t} \indicator{0 < \pi(X_u) \leq \varepsilon n} \d u \right\rvert \leq E^{K(n^2 t)+1}\\
		+ \left\lvert \sum_{k=1}^{K(n^2 t)} \left( E^k - \frac{1}{\varepsilon n} V^k(\varepsilon n) \right) \right\rvert.
	\end{multline*}
	
	If $\pi(X_{n^2 t}) > 0$, then the residual term, instead of being $E^{K(n^2 t)+1}$, is the time spent exploring nodes with label $\leq \varepsilon n$ in the $K(n^2t)$th ambient tree. In each case, one can show that this residual term does not contribute in the regime $n \to +\infty$ and then $\varepsilon \to 0$ that we are interested in, and so we only have to show that
	\[ \limsup_{n \to +\infty} \P \left( \left\lvert \sum_{k=1}^{K(n^2 t)} \left( E^k - \frac{1}{\varepsilon n} V^k(\varepsilon n) \right) \right\rvert \geq \eta n \right) \mathop{\longrightarrow}_{\varepsilon \downarrow 0} 0. \]
	
	For $y > 0$ introduce the following quantities: $m(y) = \E(E^1) - \E(V^1(y)) / y$, $\Upsilon^k(y) = E^k - V^k(y) / y - m(y)$, $\sigma(y)^2 = \E(\Upsilon^1(y)^2)$, $\overline \Upsilon^k(y) = \Upsilon^k(y) / \sigma(y)$ and
	\[ \Sigma(n, \varepsilon) = \left\lvert \frac{1}{K(n^2 t)^{1/2}} \sum_{k=1}^{K(n^2 t)} \overline \Upsilon^k(\varepsilon n) \right\rvert. \]
	
	Then the triangular inequality gives
	\[ \left\lvert \sum_{k=1}^{K(n^2 t)} \left( E^k - \frac{1}{\varepsilon n} V^k(\varepsilon n) \right) \right\rvert \leq \sigma(\varepsilon n) K(n^2 t)^{1/2} \Sigma(n, \varepsilon) + \left \lvert m(\varepsilon n) \right \rvert K(n^2 t) \]
	and so
	\begin{multline} \label{eq:tmp}
		\P \left( \left\lvert \sum_{k=1}^{K(n^2 t)} \left( E^k - \frac{1}{\varepsilon n} V^k(\varepsilon n) \right) \right\rvert \geq \eta n \right) \leq \P \left( K(n^2 t) \geq \frac{\eta n}{2 \left \lvert m(\varepsilon n) \right \rvert} \right)\\
		+ \P \left( \Sigma(n, \varepsilon) \geq \frac{\eta n}{2 \sigma(\varepsilon n) K(n^2 t)^{1/2}} \right).
	\end{multline}
	
	Let $\overline C = \sup_{y \geq 0} (y^{-1/2} \E(K(y)))$, which has been showed in the proof of Lemma~\ref{lemma:control-mean-local-time}, to be finite. Using Markov's inequality, the first term of the above upper bound is thus upper bounded by
	\begin{equation} \label{eq:K}
		\P \left( K(n^2 t) \geq \frac{\eta n}{2 \left \lvert m(\varepsilon n) \right \rvert} \right) \leq (2 / \eta) \overline C t^{1/2} \times \left \lvert m(\varepsilon n) \right \rvert.
	\end{equation}
	
	By definition we have for $y > 0$
	\[ m(y) = \frac{1}{\lambda \P(J=1)} - \frac{1}{\lambda y} \E \left( \sum_{v \in B(\Tree_1)} \indicator{\psi(v, \Tree_1) \leq y} \right) \]
	and so~\eqref{eq:m} implies that $m(y) \to 0$. In view of~\eqref{eq:K}, the first term in the right-hand side of~\eqref{eq:tmp} therefore vanishes as $n \to +\infty$. We now control the second term: let $C^* = \sup_{y \geq 0} (\sigma(y)/y^{1/2})$, which is proved to be finite in Section~\ref{sub:various}, so that
	\begin{align*}
		\P \left( \Sigma(n, \varepsilon) \geq \frac{\eta n}{2 \sigma(\varepsilon n) K(n^2 t)^{1/2}} \right) & \leq \P \left( \Sigma(n, \varepsilon) \geq \frac{\eta n^{1/2}}{2 \varepsilon^{1/2} C^* K(n^2 t)^{1/2}} \right)\\
		& \leq \P \left( \Sigma(n, \varepsilon) \geq \frac{\eta}{2 \varepsilon^{1/2} C^* \bar K} \right) + \frac{\overline C t^{1/2}}{\bar K^2}
	\end{align*}
	where the second inequality, valid for any $\bar K$, is obtained by considering the two events $\{n / K(n^2t) \geq 1/\bar K^2\}$ and $\{n / K(n^2t) \leq 1/\bar K^2\}$ and using the Markov inequality in the second case. Since the $(\overline Y^k(\varepsilon n), k \geq 1)$ are i.i.d.\ centered random variables with unit variance, the central limit theorem gives
	\[ \limsup_{n \to +\infty} \P \left( \Sigma(n, \varepsilon) \geq \frac{\eta}{2 \varepsilon^{1/2} C^* \bar K} \right) \mathop{\longrightarrow}_{\varepsilon \to 0} 0. \]
	
	Thus letting first $n \to +\infty$, then $\varepsilon \to 0$ and finally $\bar K \to +\infty$ achieves the proof.
\end{proof}

\subsection{Last step} \label{step:last}

At this point, we know that, under $\Pbf$:
\begin{enumerate}
	\item \label{prop:1} $\pi \circ X$ is equal in distribution to $(2c)^{1/2} W$ (by the fifth step);
	\item $M \circ X$ is equal in distribution to $(2 \lambda)^{1/2} W$ (by Lemmas~\ref{lemma:preliminary} and~\ref{lemma:joint-convergence});
	\item \label{prop:3} $\Lcal(\pi \circ X) = \Lcal(\E(J) M \circ X)$ (by Lemmas~\ref{lemma:joint-convergence} and~\ref{lemma:asymptotic-local-time});
\end{enumerate}

These three properties have the following consequence.

\begin{lemma}\label{lemma:price=mass}
	Under $\Pbf$, for every $t \geq 0$ we have $M(X_t) = \pi(X_t) / \E(J)$.
\end{lemma}

Before proving this lemma, let us quickly conclude the proof of Theorem~\ref{thm:main}. Fix some $t, y \geq 0$: we have to prove~\eqref{eq:limit}. If $\pi(X_t) = 0$, then $M(X_t) = 0$ by Lemma~\ref{lemma:price=mass} and~\eqref{eq:limit} holds. Otherwise, assume first that $y < \pi(X_t)$ and let $g$ be the left endpoint of the excursion of $\pi \circ X$ above $y$ straddling $t$. Then according to Corollary~\ref{cor:local-evolution}, we have $X_t([0,y]) = X_g([0,y])$. On the other hand, we have $y = \pi(X_g)$ by definition of $g$ and so $X_g([0,y]) = M(X_g)$ which is equal to $\pi(X_g) / \E(J) = y / \E(J)$ by Lemma~\ref{lemma:price=mass}. This proves that $X_t([0,y]) = y / \E(J)$ for $y < \pi(X_t)$, and since $X_t([0,y]) = M(X_t)$ for $y \geq \pi(X_t)$, which is equal to $\E(J)^{-1} \pi(X_t)$ by Lemma~\ref{lemma:price=mass}, this proves~\eqref{eq:limit} and concludes the proof of Theorem~\ref{thm:main}.

\begin{proof} [Proof of Lemma~\ref{lemma:price=mass}]
	Thanks to~\eqref{eq:canonical-decomposition-W} we can write $\pi = c \Lcal(\pi) + \bar \pi$ and $M = \lambda \Lcal(M) + \bar M$, where $(2c)^{-1/2} \bar \pi$ and $(2\lambda)^{-1/2} \bar M$ are two standard Brownian motions, and in the rest of the proof we write in order to ease the notation $\pi$ and $M$ for $\pi \circ X$ and $M \circ X$, respectively. Moreover, $\Lcal(\pi)$ is on the one hand equal in distribution to $\Lcal((c/\lambda)^{1/2} M)$ because $\pi$ is equal in distribution to $(c/\lambda)^{1/2} M$, while on the other hand we have $\Lcal(\pi) = \Lcal(\E(J) M)$. This shows that $c = \lambda \E(J)^2$ and in particular, we have $M = (c/\E(J)) \Lcal(\pi) + \bar M$.

	Fix some $a, t\ge 0$: then we can apply the optional sampling theorem (as in, e.g.,~\cite[Problem 1.3.23a]{Karatzas91:0}) for the bounded stopping time $T_a\wedge t$ and the martingales $\bar M$ and $\bar \pi$, and derive
	\begin{equation} \label{first_moment}
		\Ebf(M_{T_a\wedge t}) = (c/\E(J)) \Ebf\big(\Lcal(\pi)_{T_a\wedge t}\big) = (1/\E(J)) \Ebf(\pi_{T_a\wedge t}).
	\end{equation}

	Since $(\bar M^2 - 2 \lambda t, t \geq 0)$ and $(\bar \pi^2 - 2 c t, t \geq 0)$ are also martingales, another application of the optional sampling theorem implies
	\begin{equation} \label{second_moment}
		\Ebf\big(M_{T_a \wedge t}^2\big) = 2\lambda \Ebf(T_a\wedge t) = \frac{2\lambda}{2c} \Ebf\big(\pi_{T_a \wedge t}^2\big) = \frac{1}{\E(J)^2} \Ebf\big(\pi_{T_a \wedge t}^2\big) \le \frac{a^2}{\E(J)^2}.
	\end{equation}

	The stopped process $(M_{T_a\wedge t}, t \geq 0)$ is therefore uniformly integrable, and letting $t \to +\infty$ in~\eqref{first_moment}, we thus obtain $\Ebf(M_{T_a}) = \Ebf(\pi_{T_a}) / \E(J) = a / \E(J)$. On the other hand, letting $t \to +\infty$ in~\eqref{second_moment} and using Fatou's lemma, we obtain $\Ebf(M_{T_a}^2) \leq (a / \E(J))^2$ which implies that $M_{T_a} = a/\E(J)$.

	A calculation similar to~\eqref{second_moment} shows that the stopped Brownian motions $(\bar M_{T_a\wedge t}, t \geq 0)$ and $(\bar \pi_{T_a\wedge t}, t \geq 0)$ are uniformly integrable. Then we can apply another version of the optional sampling theorem, such as in~\cite[Problem 1.3.19 and Theorem 1.3.22]{Karatzas91:0}, and get
	\[ \Ebf\left(M_{T_a} \mid \Fcal_{t\wedge T_a}\right) = (c/\E(J)) \Ebf\big(\Lcal(\pi)_{T_a} \mid \Fcal_{t\wedge T_a} \big) +\bar M_{T_a \wedge t}. \]

	Since we have proved that $M_{T_a} = a/\E(J)$, the last display leads to
	\[ a - c \Ebf(\Lcal(\pi)_{T_a} \mid \Fcal_{t\wedge T_a}) = \E(J) \bar M_{T_a \wedge t}. \]

	The exact same reasoning shows that the left-hand side of the above display is also equal to $\bar\pi_{T_a \wedge t}$, and so $M_{t\wedge T_a}= \pi_{t\wedge T_a} / \E(J)$. Letting $a \to +\infty$ achieves the proof.
\end{proof}

\section{Discussion} \label{sec:discussion}

The main purpose of this paper is to exploit the connection between the regenerative characterization of L\'evy trees of~\cite{Weill07:0} and the present model of the limit order book. The assumptions made on $J$ in Theorem~\ref{thm:main} correspond to the simplest interesting case where this connection can be exploited, but this result should hold under more general assumptions on $J$ and $\lambda$. For instance, our arguments should readily extend to a triangular scheme where the rates at which orders are added to and removed from the book may be different, and the model's parameters depend on $n$ in a suitable way. We believe that the results of Theorem~\ref{thm:main} would still hold, with the limiting price process being a Brownian motion \emph{with drift} reflected at $0$.

A more delicate generalization consists in relaxing the assumption that $J \in \{-j^*, \ldots, 1\}$. The proof of most results goes through in this more general case, but the main problem is that for a general random variable $J$, the successive excursions above level $a$ are not i.i.d.\ anymore, which invalidates Lemma~\ref{lemma:discrete-reg}. However, the dependency between successive excursions lies in the overshoot of the price above $a$ at the beginning of each excursion above $a$, and so upon suitable moment assumptions on $J$ this dependency should be washed out in the limit. We believe that this generalization could be obtained with a suitable coupling with the case $J \in \{-j^*, \ldots, 1\}$ studied here.
\\

Further, different boundary conditions can also be considered. In~\cite{Simatos14:1} and in the additive version of~\cite{Lakner16:0} for instance, orders can be placed in the negative half-line. In~\cite{Lakner16:0} there is the additional constraint that the number of orders cannot fall below some level, say $\varepsilon n$. This is meant to model the presence of a \emph{market maker}.

In the presence of such a market maker, Theorem~\ref{thm:main} remains valid and the proofs go through. Indeed, imagine $\varepsilon n$ orders initially sit at $0$. Since these orders can only be displaced when the price is at $0$ and, while the price is at $0$, the number of orders evolves according to a critical random walk, the price process needs to accumulate of the order of $n^2$ units of local time at $0$ in order to go through this initial stack of orders. Lemma~\ref{lemma:preliminary} shows that this takes of the order of $n^4$ units of time, and so on the time scale that we are interested in, this does not happen. Pushing this reasoning a bit further actually shows that Theorem~\ref{thm:main} should remain valid as long as the initial number of orders, say $m_n$, grows to $+\infty$. Indeed, in this case after accumulating $m_n^2$ units of local time at $0$ these orders will have only moved by a constant distance, and it would thus take $n m_n^2$ units of local time at $0$, which take about $n^2 m_n^4 \gg n^2$ units of normal time to accumulate, to have them moved by a distance of the order of $n$.

On the other hand, when orders can be placed on the negative half-line and there is no market maker, then we conjecture that the price process should converge to a Brownian motion (without reflection), say $\tilde W$, and the measure-valued process should converge to the process having constant density $1/\E(J)$ with respect to Lebesgue measure restricted to $[I_t, \tilde W_t]$ with $I_t = \inf_{[0,t]} \tilde W$. The key observation is indeed that if, in this ``free'' case, one reflects the measure-valued process by considering $I^\pi$, the past infimum of the price process, as the origin of space and collapsing all the orders below $I^\pi$ at $I^\pi$, then one precisely gets the model studied here. Thus the only thing left to prove would be that $I^\pi$ converges to the local time at $0$ of the reflected price process.
\\

Let us finally mention that we have focused here on the case $\E(J) > 0$. When $\E(J) < 0$, under minor moment assumptions on $J$ the probability $\P(\psi^*(B(\Tree_1)) > u)$ decays \emph{exponentially fast}, since for this to happen one needs the supremum of a random walk with a negative drift to be large (see for instance Theorem~$2$ in~\cite{addario-berry11:0}). This is in sharp contrast with the polynomial decay proved in Lemma~\ref{lemma:tail-psi} when $\E(J) > 0$, and it implies, when $\E(J) < 0$, that $\pi \circ X$ under $\P^n$ converges weakly to $0$ (since one would need to see an exponential number of excursions before seeing a macroscopic one). Note that the case $\E(J) < 0$ with a different boundary condition (see discussion above) has been studied in~\cite{Lakner16:0} via a completely different approach. The fact that $\pi \circ X$ under $\P^n$ converges to $0$ means, in terms of the free process studied in~\cite{Lakner16:0}, that the limiting price process is increasing (see Proposition $9.12$ there).

To conclude, we note that the case $\E(J) = 0$, which is in some sense the true critical case where both the offspring and displacement distributions of $\Tree_1$ are critical, remains open.

\appendix

\section{Results on a branching random walk with a barrier} \label{appendix}

We prove in this section the various results on $B(\Tree_1)$ that have been used in the proof of Theorem~\ref{thm:main}. These results may also be of independent interest, see for instance~\cite{Durrett91:0} and~\cite{Kesten94:0} where closely related results are proved for $\Tree_1$. Note that we consider the case of a geometric offspring distribution, but the arguments above actually work for any offspring distribution with finite exponential moments. With more care, they could probably be extended to a more general setting.

\subsection{Preliminary results}

Let in the sequel $Z_m = \sum_{v \in \Tree_1} \indicator{\lvert v \rvert = m}$ for $m \geq 1$ be the number of nodes at depth $m$ in $\Tree_1$, so that $(Z_m, m \geq 1)$ is a Galton--Watson branching process with geometric offspring distribution with parameter $1/2$, $h(\Tree_1)$ is its extinction time and $\lvert \Tree_1 \rvert$ is its total progeny. By induction one easily obtains
\begin{equation} \label{eq:variance-Z}
	\E \left[ (Z_m-1)^2 \right] = 2m, \ m \geq 1.
\end{equation}

Moreover, it is well known that there exists a finite constant $C_S > 0$ such that
\begin{equation} \label{eq:tail-behavior-size+height}
	\P \left( \lvert \Tree_1 \rvert \geq u \right) \mathop{\sim}_{u \to +\infty} C_S / u^{1/2} \ \text{ and } \ P \left( h(\Tree_1) \geq u \right) \mathop{\sim}_{u \to +\infty} 1 / u,
\end{equation}
see for instance~\cite[Theorem~$23$]{Aldous93:0} where these estimates are established for any finite variance Galton--Watson process. Most of the times upper and lower bounds will be enough, and we will for instance often write
\[ 1 / (C u^{1/2}) \leq \P \left( \lvert \Tree_1 \rvert \geq u \right) \leq C / u^{1/2} \ \text{ and } \ 1 / (Cu) \leq P \left( h(\Tree_1) \geq u \right) \leq C / u. \]

We will also need the existence of a finite constant $C > 0$ such that for every $u, m \geq 1$,
\begin{equation} \label{eq:GW}
	\E \left( Z_m \mid \lvert \Tree_1 \rvert > u \right) \leq C m \ \text{ and } \ \E \left( Z_m \mid h(\Tree_1) > u \right) \leq C m.
\end{equation}

The first bound can be found in, e.g.,~\cite[Theorem~$1.13$]{Janson06:1}, where it is proved for any finite variance Galton--Watson process. The second bound is very natural in view of the first one, since the trees conditioned on having a large size or a large height are known to have the same scaling limits, but we could not find a precise reference for it and we therefore provide a proof. The following proof is due to Igor Kortchemski, to whom we are grateful for sharing it with us.

\begin{proof}[Proof of the second bound in~\eqref{eq:GW}]
	Since $\P(h(\Tree_1) > u) \geq C/u$, we have
	\[ \E(Z_m \mid h(\Tree_1) > u) \leq Cu \E(Z_m ; h(\Tree_1) > u). \]
	
	If $u \leq 2m$, then we simply use $\E(Z_m ; h(\Tree_1) > u) \leq \E(Z_m) = 1$ to get $\E(Z_m \mid h(\Tree_1) > u) \leq Cu \leq 2 Cm$. Assume now that $u \geq 2m$. Then given $Z_m = z$, in order to have $h(\Tree_1) > u$ at least one of the $z$ subtrees rooted at depth $m$ must have height $> u-m$. Using the branching property we thus get
	\[ \E(Z_m ; h(\Tree_1) > u) = \E \left[ Z_m \left(1 - q(u-m)^{Z_m} \right) \right] \]
	with $q(x) = 1 - \P(h(\Tree_1) > x)$. Using $\E(Z_m) = 1$ we obtain
	\[ \E(Z_m ; h(\Tree_1) > u) = 1 - \E \left[ Z_m q(u-m)^{Z_m} \right] \]
	and since $q(x) \geq 1 - C/x$, this yields
	\begin{align*}
		\E(Z_m ; h(\Tree_1) > u) & \leq 1 - \E \left[ Z_m \left(1 - \frac{C}{u-m} \right)^{Z_m} \right]\\
		& \leq 1 - \E \left[ Z_m \left(1 - \frac{CZ_m}{u-m} \right) \right]\\
		& = \frac{C}{u-m} \E \left(Z^2_m\right),
	\end{align*}
	where we have used $(1-x)^z \geq 1 - zx$ for the second inequality. Since $\E(Z^2_m) \leq Cm$ by~\eqref{eq:variance-Z}, we finally get
	\[ \E(Z_m \mid h(\Tree_1) > u) \leq \frac{Cum}{u-m} \leq Cm, \]
	using for the second inequality $u/(u-m) \leq 2$ when $u \geq 2m$. The proof is complete.
\end{proof}

Let in the rest of this section $S = (S_m, m \geq 0)$ be a random walk started at $0$ and with step distribution $J$, independent from $\Tree_1$, and $\underline S_m = \min_{0 \leq k \leq m} S_k$.

\begin{lemma}\label{lemma:number-killed}
	We have $\E(\lvert \Kcal(\Tree_1) \rvert^2) < +\infty$ and
	\begin{equation} \label{eq:mean-killed}
		\E \left( \lvert \Kcal(\Tree_1) \rvert \right) = 1 - \frac{\E(J)}{\P(J=1)}.
	\end{equation}
\end{lemma}

\begin{proof}
	To compute the mean number of killed nodes, we write
	\[ \lvert \Kcal(\Tree_1) \rvert = \sum_{m \geq 1} \sum_{v \in \Tree_1: \lvert v \rvert = m} f(v) \ \text{ with } \ f(v) = \indicator{\psi(v, \Tree_1) \leq 0, \psi(v_1, \Tree_1), \ldots, \psi(v_{m - 1}, \Tree_1) \geq 1}. \]
	
	Thus, taking the mean and using the fact that labels are independent from the genealogical structure, we obtain
	\[ \E \left( \lvert \Kcal(\Tree_1) \rvert \right) = \sum_{m \geq 1} \P \left( S_m < 0, \underline S_{m-1} \geq 0 \right) \E(Z_m). \]
	
	Since the genealogical structure $Z$ of $\Tree_1$ is a critical Galton--Watson process, we have $\E(Z_m) = 1$ which gives $\E(\lvert \Kcal(\Tree_1) \rvert) = \P(\underline S_\infty < 0)$. Since $J \in \left\{-j^*, -j^* + 1, \ldots, 0, 1\right\}$, $S$ is, in the terminology of~\cite{Brown10:0}, a skip-free (to the right) random walk with positive drift. In particular, Corollary~$1$ in this reference implies that $\P(\underline S_\infty \geq 0) = \E(J) / \P(J = 1)$ which gives~\eqref{eq:mean-killed}.
	
	As for the second moment, we define $v \wedge v'$ for $v, v' \in \Tree_1$ as the most recent common ancestor of $v$ and $v'$, and write $\lvert \Kcal(\Tree_1) \rvert^2 = \lvert \Kcal(\Tree_1) \rvert + \Sigma$, so that we only have to prove that $\E(\Sigma) < +\infty$, with
	\[ \Sigma = \sum_{\substack{M \geq 1 \\ m, m' \geq M }} \sum_{V: \lvert V \rvert = M} \sum_{\substack{v: \lvert v \rvert = m \\ v': \lvert v' \rvert = m'}} f(v) f(v') \indicator{v \wedge v' = V, v \not = v'}. \]
	
	Let $M \geq 1$, $m, m' \geq M$ and $V, v, v' \in \Tree_1$ with $\lvert V \rvert = M$, $\lvert v \rvert = m$, $\lvert v' \rvert = m'$, $v \wedge v' = V$ and $v \not = v'$. If $v$ is an ancestor of $v'$ (or the other way around), then $f(v) f(v') = 0$. Otherwise, $m, m' > M$ and the paths from the root to $v$ and $v'$ coincide on the first $M$ steps and are independent afterwards, on the $m-M$ and $m'-M$ remaining steps, respectively. Thus in this case, if $S'$ is an independent copy of $S$ we have
	\begin{align*}
		\E\left( f(v) f(v') \mid Z \right) & = \P \Big( \underline S_{m-1} \geq 0, S_m < 0,\\
		& \hspace{10mm} S_M +S'_1, \ldots, S_M + S'_{m'-M-1} \geq 0, S_M + S'_{m'-M} < 0 \Big)\\
		& = \E \left[ g(m-M, S_M) g(m'-M, S_M) ; \underline S_M \geq 0 \right]
	\end{align*}
	where $g(i,s) = \P \left( \underline S_{i-1} \geq -s, S_i < -s \right)$ for $i \geq 1$ and $s \in \N$. Defining $g(0,s) = 0$, we therefore get that $\E(\Sigma)$ is upper bounded by
	\[ \sum_{\substack{M \geq 1 \\ m, m' \geq M }} \E \left( g(m-M, S_M) g(m'-M, S_M) \right) \, \E \left( \sum_{V: \lvert V \rvert = M} \sum_{\substack{v: \lvert v \rvert = m\\v': \lvert v'\rvert = m'}} \indicator{v \wedge v' = V} \right). \]
	
	Since by the branching property, the subtrees rooted at nodes at depth $M$ in the tree are i.i.d., independent from the number $Z_M$ of nodes at depth $M$, and since further $\E(Z_M) = 1$, we have
	\[ \E \left( \sum_{V: \lvert V \rvert = M} \sum_{\substack{v: \lvert v \rvert = m\\v': \lvert v'\rvert = m'}} \indicator{v \wedge v' = V} \right) = \E \left( \sum_{\substack{v: \lvert v \rvert = m-M\\v': \lvert v'\rvert = m'-M}} \indicator{v \wedge v' = \emptyset} \right). \]
	
	To count the number of nodes at depths $m{-}M$ and $m'{-}M$ with most recent common ancestor the root, we can pick two distinct children of the root and then count the number of nodes at depth $m{-}M{-}1$ and $m'{-}M{-}1$ in each subtree, so that
	\[ \sum_{\substack{v: \lvert v \rvert = m-M\\v': \lvert v'\rvert = m'-M}} \indicator{v \wedge v' = \emptyset} = \sum_{u, u': \lvert u \rvert = \lvert u' \rvert = 1} Z(u, m-M-1) Z(u', m'-M-1) \]
	where $Z(w, i)$ is the number of nodes at depth $i$ in the subtree of $\Tree$ rooted at $w \in \Tree$. Thus taking the mean and noting that the number of distinct pairs of children of the root is equal in distribution to $Z_1(Z_1-1)$ which has mean $2$, we obtain
	\[ \E \left( \sum_{V: \lvert V \rvert = M} \sum_{\substack{v: \lvert v \rvert = m\\v': \lvert v'\rvert = m'}} \indicator{v \wedge v' = V, v \not = v'} \right) = 2 \E \left( Z_{m-M-1} Z_{m'-M-1} \right) = 2 \E(Z_{m^*}) \]
	with $m^* = \min(m-M,m'-M) + 1$. Using that $g(0,s) = 0$, upper bounding $m^*$ by $3((m-M) + (m'-M))$ when $m, m' > M$, changing variables in the sum and using~\eqref{eq:variance-Z}, we get
	\[ \E\left( \Sigma \right) \leq 3 \sum_{M, m, m' \geq 1} \E \left( g(m, S_M) g(m', S_M) \right) (m + m'). \]
	
	We have by definition $g(m,s) = \P(\underline S_{m-1} \geq -s, S_m < -s)$, so that for any $\kappa > 0$,
	\begin{equation} \label{eq:exp-decay}
		g(m,s) \leq \P \left( S_m < -s \right) \leq e^{-\kappa s} \left[ \E(e^{-\kappa J}) \right]^m \leq e^{-\kappa s} \left[ \E(e^{-2 \kappa J}) \right]^{m/2}.
	\end{equation}
	
	Since $\E(J) > 0$, we can choose $\kappa > 0$ such that $\beta = \E(e^{-2\kappa J}) < 1$, and so we get the bound
	\[ \E\left( \Sigma \right) \leq 3 \sum_{M, m, m' \geq 1} \E \left( e^{-2\kappa S_M} \beta^{(m+m')/2} \right) (m + m') = 3 \sum_{M \geq 1} \beta^M \sum_{m \geq 1} m^2 \beta^m. \]
	
	Since $\beta < 1$, these two sums are finite, which achieves to prove that $\lvert \Kcal(\Tree_1) \rvert$ has a finite second moment.
\end{proof}

\begin{lemma} \label{lemma:tail-behavior-h}
	As $u \to +\infty$, we have $u \P(h(B(\Tree_1)) \geq u) \to \E(J) / \P(J=1)$.
\end{lemma}

\begin{proof}
	Define for simplicity $\kappa = \lvert \Kcal(\Tree_1) \rvert$ and let $(v^B_k, k = 1, \ldots, \kappa)$ be the $\kappa$ killed nodes in $\Tree_1$, and $(\Tree^{(k)}, k = 1, \ldots, \kappa)$ be the subtrees attached to them. Then
	\[ h(\Tree_1) = \max \left( h(B(\Tree_1)), \lvert v^B_1 \rvert + h(\Tree^{(1)}) - 1, \ldots, \lvert v^B_\kappa \rvert + h(\Tree^{(\kappa)}) - 1 \right) \]
	so that
	\begin{multline*}
		\P \left( h(\Tree_1) \geq u \right) = \P \left( h(B(\Tree_1)) \geq u \right)\\
		+ \P \left( \exists k \in \{1, \ldots, \kappa\}: h(B(\Tree_1)) < u \text{ and } h(\Tree^{(k)}) \geq u + 1 - \lvert v^B_k \rvert\right).
	\end{multline*}

	Next, we observe that conditionally on $B(\Tree_1)$, the $(h(\Tree^{(k)}), k = 1, \ldots, \kappa)$ are i.i.d.\ with common distribution $h(\Tree_1)$. Defining $H(u) = \P(h(\Tree_1) \geq u)$, we thus obtain
	\begin{multline*}
		\P \left( h(\Tree^{(k)}) \geq u + 1 - \lvert v^B_k \rvert \text{ for some } k \in \{1, \ldots, \kappa\} \mid B(\Tree_1) \right)\\
		= 1 - \prod_{k=1}^\kappa \left( 1 - H(u+1-\lvert v^B_k\rvert) \right)
	\end{multline*}
	and consequently,
	\[ H(u) = \E \left( Y(u) ; h(B(\Tree_1)) < u \right) + \P \left( h(B(\Tree_1)) \geq u \right) \]
	with
	\[ Y(u) = 1 - \prod_{k=1}^\kappa \left( 1 - H(u+1-\lvert v^B_k\rvert) \right). \]
	
	It follows from~\eqref{eq:tail-behavior-size+height} that the random variable $u Y(u) \indicator{h(B(\Tree_1)) < u}$ converges almost surely as $u \to +\infty$ to $\kappa$. If we had uniform integrability, then we would obtain
	\[ u \P \left( h(B(\Tree_1)) \geq u \right) = u H(u) - \E \left( u Y(u) ; h(B(\Tree_1)) < u \right) \mathop{\longrightarrow}_{u \to +\infty} 1 - \E(\kappa) \]
	which would prove the result by~\eqref{eq:mean-killed}. Thus it remains to show that the family of random variables $(u Y(u) \indicator{h(B(\Tree_1)) < u}, u \geq 0)$ is uniformly integrable: it is enough to show that $\sup_{u \geq 1} \E(u^2 Y(u)^2)$ is finite. Let $V^B = \max_{k = 1, \ldots, \kappa} \lvert v^B_k \rvert$: since $Y(u) \leq 1$ and $Y$ is increasing in each $\lvert v^B_k \rvert$, we have
	\[ \E \left( Y(u)^2 \right) \leq \P \left( V^B \geq u/2 \right) + \E \left[ \left(1 - (1-H(u/2+1))^\kappa \right)^2 \right]. \]
	
	In the event $V^B \geq u/2$, we have $N \geq 1$ where $N$ is the number of nodes $v \in \Tree_1$ that satisfy $\lvert v \rvert \geq u/2$ and $\psi(v, \Tree_1) \leq 0$. Using Markov inequality, we therefore get
	\[ \P \left( V^B \geq u/2 \right) \leq \E \left( \sum_{v \in \Tree_1} \indicator{\lvert v \rvert \geq u/2, \psi(v, \Tree_1) \leq 0} \right) = \sum_{m \geq u/2} \P \left( S_m \leq 0 \right). \]
	
	Using $1 - (1-x)^y \leq xy$ for $y \geq 0$, we get on the other hand
	\[ \E \left[ \left(1 - (1-H(u/2+1))^\kappa \right)^2 \right] \leq H(u/2+1)^2 \E \left(\kappa^2 \right) \]
	so that finally,
	\[ u^2 \E \left( Y(u)^2 \right) \leq u^2 \sum_{m \geq u/2} \P \left( S_m \leq 0 \right) + \left( u \P \left( h(\Tree_1) \geq u/2+1 \right) \right)^2 \E \left(\kappa^2 \right). \]
	
	Since the probability $\P(S_m \leq 0)$ decays exponentially fast as $m \to +\infty$ by~\eqref{eq:exp-decay}, the first term of the above upper bound is bounded in $u$. The second term being also bounded in $u$ by~\eqref{eq:tail-behavior-size+height} and Lemma~\ref{lemma:number-killed}, the proof is complete.
\end{proof}

\begin{lemma} \label{lemma:tail-psi}
	As $u \to +\infty$, we have
	\[ u \P(\psi^*(B(\Tree_1)) \geq u) \to \frac{(\E(J))^2}{\P(J=1)}. \]
\end{lemma}

\begin{proof}
	Let
	\[ \overline \psi = \sup_{v \in \Tree_1} \lvert v \rvert^{-2/3} \left \lvert \psi(v, \Tree_1) - \lvert v \rvert \E(J) \right \rvert, \]
	so that for any $\varrho > 0$,
	\begin{multline} \label{eq:decomposition-psi}
		\P \left( \psi^*(B(\Tree_1)) \geq u, \overline \psi \leq u^\varrho \right) \leq \P \left( \psi^*(B(\Tree_1)) \geq u \right)\\
		\leq \P \left( \psi^*(B(\Tree_1)) \geq u, \overline \psi \leq u^\varrho \right) + u^{-12 \varrho} \E \left( \overline \psi^{12} \right).
	\end{multline}
	
	We show that $\E(\overline \psi^{12})$ is finite. By upper bounding the supremum by the sum, we get
	\begin{align*}
		\E \left( \overline \psi^{12} \right) & \leq \sum_{m \geq 1} \frac{1}{m^8} \E \left( \sum_{v \in \Tree_1, \lvert v \rvert = m} \left \lvert \psi(v, \Tree_1) - m \E(J) \right \rvert^{12} \right)\\
		& = \sum_{m \geq 1} \frac{1}{m^8} \E\left[ \left \lvert Y_1 + \cdots + Y_m \right \rvert^{12} \right]
	\end{align*}
	where $(Y_i)$ are i.i.d.\ centered random variables with distribution $J - \E(J)$ and where, in order to derive the last equality, we used the independence in $\Tree_1$ between the genealogical structure and the labels. The central limit theorem implies that $\lvert Y_1 + \cdots + Y_m \rvert / m^{1/2}$ converges weakly, and since the $Y_k$'s are bounded, all the moments of this random variable are bounded uniformly in $m$, so that by uniform integrability we can write $\E\left[\lvert Y_1 + \cdots + Y_m\rvert^{12} \right] \leq C m^6$ for all $m \geq 1$ and some finite constant $C$, independent from $m$. This gives $\E(\overline \psi^{12}) \leq C \sum_{m \geq 1} m^{-2}$ which is finite.
	\\
	
	We now derive an upper bound on the term $\P(\psi^*(B(\Tree_1)) \geq u, \overline \psi \leq u^\varrho)$ in~\eqref{eq:decomposition-psi}. In the event $\{ \psi^*(B(\Tree_1)) \geq u \}$, there exists $v^* \in B(\Tree_1)$ such that $u \leq \psi(v^*, \Tree_1) \leq \lvert v^* \rvert$. Moreover, by definition of $\overline \psi$ we have $\psi(v^*, \Tree_1) \leq \lvert v^* \rvert \E(J) + \lvert v^* \rvert^{2/3} \overline \psi$, and so when both events $\{ \psi^*(B(\Tree_1)) \geq u \}$ and $\{ \overline \psi \leq u^\varrho \}$ hold, there exists $v^* \in B(\Tree_1)$ such that
	\[ u \leq \lvert v^* \rvert \E(J) + \lvert v^* \rvert^{2/3} u^\varrho \leq \lvert v^* \rvert \E(J) + \lvert v^* \rvert^{\varrho + 2/3}, \]
	which can be rewritten as $u \leq \phi(\lvert v^* \rvert \E(J))$ where $\phi(x) = x + (x/\E(J))^{\varrho + 2/3}$ for $x \geq 0$. If $\phi^{-1}$ stands for its inverse, we therefore have
	\[ \P \left( \psi^*(B(\Tree_1)) \geq u, \overline \psi \leq u^\varrho \right) \leq \P \left( h(B(\Tree_1)) \geq \phi^{-1}(u) / \E(J) \right) \]
	so that plugging these inequalities in~\eqref{eq:decomposition-psi}, we obtain
	\[ u \P \left( \psi^*(B(\Tree_1)) \geq u \right) \leq \frac{C}{u^{12 \varrho - 1}} + u \P \left( h(B(\Tree_1)) \geq \phi^{-1}(u) / \E(J) \right). \]
	
	For $\varrho > 1/12$ the first term of the above upper bound vanishes, while for $\varrho < 1/3$ we have $\phi^{-1}(u) \sim u$ as $u \to +\infty$ and so the second one goes to $(\E(J))^2 / \P(J=1)$ by Lemma~\ref{lemma:tail-behavior-h}. Thus choosing $1/12 < \varrho < 1/3$ we obtain
	\[ \limsup_{u \to +\infty} u \P \left( \psi^*(B(\Tree_1)) \geq u \right) \leq \frac{(\E(J))^2}{\P(J=1)}. \]
	
	Starting from the lower bound in~\eqref{eq:decomposition-psi} a corresponding lower bound can be proved using the same arguments, which completes the proof.
\end{proof}

\subsection{Various results} \label{sub:various} We now provide the proof of the various results that have been used in the proof of Theorem~\ref{thm:main}.

\begin{proof}[Result needed in the proof of Lemma~\ref{lemma:control-mean-local-time}]
	To complete the proof of Lemma~\ref{lemma:control-mean-local-time}, we need to show that there exists a finite constant $C > 0$ such that for every $u > 0$,
	\begin{equation} \label{eq:estimate-ell}
		\P \left( \tau(\Tree_1) \geq u \right) \geq C u^{-1/2}.
	\end{equation}
	Indeed, we have
	\begin{align*}
		\P \left( \tau(\Tree_1) \geq u \right) & \geq \P \left( \lvert \Tree_1 \rvert \geq u, \Kcal(\Tree_1) = \emptyset \right)\\
		& = \P \left( \lvert \Tree_1 \rvert \geq u \right) - \P \left( \lvert \Tree_1 \rvert \geq u, \lvert \Kcal(\Tree_1) \rvert \geq 1 \right).
	\end{align*}
	We have $\P(\lvert \Tree_1 \rvert \geq u) \geq C u^{-1/2}$ by~\eqref{eq:tail-behavior-size+height}, while
	\[ \P \left( \lvert \Tree_1 \rvert \geq u, \lvert \Kcal(\Tree_1) \rvert \geq 1 \right) \leq \sqrt {\P \left( \lvert \Tree_1 \rvert \geq u \right) \E \left( \lvert \Kcal(\Tree_1) \rvert \right)} \leq C u^{-1/4}, \]
	where the first inequality comes from using first Cauchy-Schwarz inequality and then Markov inequality, and the second inequality comes from~\eqref{eq:tail-behavior-size+height} and the fact that $\E(\lvert \Kcal(\Tree_1) \rvert)$ is finite by Lemma~\ref{lemma:number-killed}. We thus get
	\[ \P \left( \tau(\Tree_1) \geq u \right) \geq C (u^{-1/2} - u^{-1/4}) \geq C u^{-1/2} \]
	which concludes the proof.
\end{proof}

\begin{proof}[Results needed in the proof of Lemmas~\ref{lemma:lebesgue-measure} and~\ref{lemma:asymptotic-local-time}]
	To complete the proof of Lemma~\ref{lemma:lebesgue-measure}, we need to show that there exists a finite constant $C > 0$ such that for every $p \geq 0$,
	\begin{equation} \label{eq:estimate-lebesgue-measure}
		\E \left( \sum_{v \in B(\Tree_1)} \indicator {\psi(v, \Tree_1) \leq p} \right) \leq C p,
	\end{equation}
	while in the proof of Lemma~\ref{lemma:asymptotic-local-time} we need to prove that
	\begin{equation} \label{eq:m}
		\frac{1}{y} \E \left( \sum_{v \in B(\Tree_1)} \indicator{\psi(v, \Tree_1) \leq y} \right) \mathop{\longrightarrow}_{y \to +\infty} \frac{1}{\P(J=1)}.
	\end{equation}
	
	Since~\eqref{eq:m} implies~\eqref{eq:estimate-lebesgue-measure} we prove~\eqref{eq:m}. We have
	\[ \E \left( \sum_{v \in B(\Tree_1)} \indicator{\psi(v, \Tree_1) \leq y} \right) = \sum_{m \geq 1} \P \left( S_m \leq y, \underline S_m \geq 0 \right). \]

	As $m \to +\infty$, $S_m / m$ conditionally on $\{\underline S_m \geq 0\}$ converges to $\E(J)$. One can therefore show that
	\[ \sum_{m \geq 1} \P \left( S_m \leq y, \underline S_m \geq 0 \right) = \sum_{m = 1}^{y / \E(J)} \P \left( \underline S_m \geq 0 \right) + o(y). \]

	Since $\P \left( \underline S_m \geq 0 \right) \to \P \left( \underline S_\infty \geq 0 \right) = \E(J) / \P(J=1)$, this gives
	\[ \sum_{m \geq 1} \P \left( S_m \leq y, \underline S_m \geq 0 \right) = \frac{y}{\E(J)} \P\left( \underline S_\infty \geq 0 \right) + o(y) = \frac{y}{\P(J=1)} + o(y) \]
	which proves~\eqref{eq:m}.
\end{proof}

\begin{proof}[Results needed in the proof of Lemma~\ref{lemma:D}]
	To complete the proof of Lemma~\ref{lemma:D}, we must show that there exists a finite constant $C > 0$ such that for every $p \geq 1$ and every $\kappa, u > 0$
	\begin{equation} \label{eq:estimate-D-1}
		\P \left( \sum_{v \in \Tree_1} \indicator{\psi(v, \Tree_1) \leq p} \geq \kappa \mid \tau(\Tree_1) > u \right) \leq \frac{C p^2}{\kappa}
	\end{equation}
	and
	\begin{equation} \label{eq:estimate-D-2}
		\P \left( \sum_{v \in \Tree_1} \indicator{\psi(v, \Tree_1) = p} \geq \kappa \mid \tau(\Tree_1) > u \right) \leq \frac{C p}{\kappa}.
	\end{equation}

	Note that~\eqref{eq:estimate-D-2} implies~\eqref{eq:estimate-D-1} by summation over $p$, so we only need prove~\eqref{eq:estimate-D-2}. Let $N_p = \sum_{v \in \Tree_1} \indicator{\psi(v, \Tree_1) = p}$: to control $\P(N_p \geq \kappa \mid \tau(\Tree_1) > u)$ we start by writing
	\[ \P \left( N_p \geq \kappa \mid \tau(\Tree_1) > u \right) = \frac{\P \left( N_p \geq \kappa, \tau(\Tree_1) > u \right)}{\P \left( \tau(\Tree_1) > u \right)} \leq \frac{C}{\kappa} \E \left( N_p \mid \lvert \Tree_1 \rvert > u/2 \right) \]
	using~\eqref{eq:estimate-ell}, $\tau(\Tree_1) \leq 2 \lvert \Tree_1 \rvert$ by~\eqref{eq:formula-tau}, Markov inequality and~\eqref{eq:tail-behavior-size+height} to derive the inequality. Conditioning on the genealogical structure leads as before to
	\[ \E \left( N_p \mid \lvert \Tree_1 \rvert > u/2 \right) = \sum_{m \geq 1} \E \left( Z_m \mid \lvert \Tree_1 \rvert > u/2 \right) \P \left( S_m = p \right) \]
	and combining the two previous displays with~\eqref{eq:GW}, we end up with
	\[ \P \left( N_p \geq \kappa \mid \tau(\Tree_1) > u \right) \leq \frac{C}{\kappa} \sum_{m \geq 1} m \P \left( S_m = p \right) = \frac{C}{\kappa} \E \left( \sum_{m \geq 1} m \indicator{S_m = p} \right). \]

	Since $S$ takes a geometric number of times the value $p$ at times around $p / \E(J)$, the term $\E(\sum_{m \geq 1} m \indicator{S_m = p})$ is of the order of $p$ when $p$ grows large, which concludes the proof.
\end{proof}

\begin{proof}[Results needed in the proof of Lemma~\ref{lemma:T}]
	To complete the proof of Lemma~\ref{lemma:T} we need to prove the two following results:
	\begin{equation} \label{eq:estimate-T-1}
		\limsup_{n \to +\infty} \P \left( \psi^*(B(\Tree_1)) < (\varepsilon^2 + \varepsilon) n \mid \psi^*(B(\Tree_1)) \geq \varepsilon n \right) \mathop{\longrightarrow}_{\varepsilon \to 0} 0
	\end{equation}
	and
	\begin{equation} \label{eq:estimate-T-2}
		\limsup_{n \to +\infty} \P \left( \lvert \Tree_1\rvert > \varepsilon n^2 \mid \psi^*(B(\Tree_1)) \geq \varepsilon n \right) \mathop{\longrightarrow}_{\varepsilon \to 0} 0.
	\end{equation}

	Note that~\eqref{eq:estimate-T-1} follows immediately from Lemma~\ref{lemma:tail-psi}. As for~\eqref{eq:estimate-T-2}, we have
	\begin{align*}
		\P \left( \lvert \Tree_1\rvert > \varepsilon n^2 \mid \psi^*(B(\Tree_1)) \geq \varepsilon n \right) & = \frac{\P \left( \lvert \Tree_1\rvert > \varepsilon n^2, \psi^*(B(\Tree_1)) \geq \varepsilon n \right)}{\P \left( \psi^*(B(\Tree_1)) \geq \varepsilon n \right)}\\
		& \leq C \varepsilon n \P \left( \lvert \Tree_1\rvert > \varepsilon n^2 \right)
	\end{align*}
	where the last inequality results from Lemma~\ref{lemma:tail-psi}. Invoking~\eqref{eq:tail-behavior-size+height} thus gives the desired~\eqref{eq:estimate-T-2}.
\end{proof}

\begin{proof}[Results needed in the proof of Lemma~\ref{lemma:g-d}]
	To complete the proof of Lemma~\ref{lemma:g-d}, we must show that there exists a finite constant $C > 0$ such that for every $u, \kappa > 0$ and $p \geq 1$,
	\begin{equation} \label{eq:estimate-g-d-1}
		\P \left( \sum_{v \in \Tree_1} \indicator{\psi(v, \Tree_1) \leq p} \geq \kappa \mid \psi^*(B(\Tree_1)) > u \right) \leq \frac{C p^2}{\kappa}
	\end{equation}
	and
	\begin{equation} \label{eq:estimate-g-d-2}
		\P \left( \sum_{v \in \Tree_1} \indicator{\psi(v, \Tree_1) = p} \geq \kappa \mid \psi^*(B(\Tree_1)) > u \right) \leq \frac{C p}{\kappa}.
	\end{equation}

	As for~\eqref{eq:estimate-D-1} and~\eqref{eq:estimate-D-2} we only need prove~\eqref{eq:estimate-g-d-2}: combining $\psi^*(B(\Tree_1)) \leq h(\Tree_1)$, Lemma~\ref{lemma:tail-psi} and~\eqref{eq:tail-behavior-size+height}, we obtain
	\begin{multline*}
		\P \left( \sum_{v \in \Tree_1} \indicator{\psi(v, \Tree_1) = p} \geq \kappa \mid \psi^*(B(\Tree_1)) > u \right)\\
		\leq C \P \left( \sum_{v \in \Tree_1} \indicator{\psi(v, \Tree_1) = p} \geq \kappa \mid h(\Tree_1) > u \right).
	\end{multline*}

	From there,~\eqref{eq:estimate-g-d-2} can be proved by repeating verbatim the proof of~\eqref{eq:estimate-D-2} with the following caveat: one needs to replace the conditioning on $\lvert \Tree_1 \rvert$ by a conditioning on $h(\Tree_1)$, and thus to use the second bound in~\eqref{eq:GW} instead of the first one.
\end{proof}

\begin{proof} [Result needed in the proof of Lemma~\ref{lemma:varepsilon}]
	To complete the proof of Lemma~\ref{lemma:varepsilon}, we need to prove that
	\begin{equation} \label{eq:kcal}
		\sup_{u \geq 0} \E \left( \lvert \Kcal(\Tree_1) \rvert \mid \psi^*(B(\Tree_1)) > u \right) < +\infty.
	\end{equation}
	
	The same arguments as in the previous proof apply and show that
	\[ \E \left( \lvert \Kcal(\Tree_1) \rvert \mid \psi^*(B(\Tree_1)) > u \right) \leq C \E \left( \lvert \Kcal(\Tree_1) \rvert \mid h(\Tree_1) > u \right). \]
	
	Conditioning on the genealogical structure $(Z_m, m \geq 1)$, we get
	\begin{align*}
		\E \left( \lvert \Kcal(\Tree_1) \rvert \mid h(\Tree_1) > u \right) & = \sum_{m \geq 1} \E(Z_m \mid h(\Tree_1) > u) \P \left( S_m < 0, \underline S_{m-1} \geq 0 \right)\\
		& \leq C \sum_{m \geq 1} m \P \left( S_m < 0 \right)
	\end{align*}
	using~\eqref{eq:GW} for the last inequality. Since $\P(S_m < 0)$ decays exponentially fast in $m$ by~\eqref{eq:exp-decay}, the sum $\sum_{m \geq 1} m \P(S_m < 0)$ is finite, which gives~\eqref{eq:kcal}.
\end{proof}

\begin{proof}[Result needed in the proof of Lemma~\ref{lemma:asymptotic-local-time}] \label{subsub:C^*}
	To complete the proof of Lemma~\ref{lemma:asymptotic-local-time} we need to prove that the constant $C^*$ defined there is finite. Let $N(y)$ be the number of nodes in $B(\Tree_1)$ with label $\leq y$: then going back to the definition of $C^*$, we see that we have to prove that $\sup_y (\Var(\Scal(N(y))) / y^3)$ is finite, where $\Var(Y)$ denotes the variance of a real valued random variable $Y$. Thanks to~\eqref{eq:bounds-exploration}, we only have to show that $\sup_y (\Var(N(y)) / y^3)$ is finite. Further, using the same estimates as in the proof of Lemma~\ref{lemma:tail-psi} we can show that $N(y)$ behaves like the number of nodes in $\Tree_1$ at depth $\leq y / \E(J)$, and in particular $\Var(N(y))$ is of the order of $\Var(Z_1 + \cdots + Z_{y/\E(J)})$. Thus in order to prove that $C^* < +\infty$, we only have to prove that $\Var(Z_1 + \cdots + Z_n)$ grows at most like $n^3$. Let
	\[ v_n = \E \left( (Z_{n+1} - 1) \sum_{k=1}^n (Z_k-1) \right). \]

	Then conditioning on $(Z_k, k \geq n)$, we obtain
	\[ v_n = \E \left( (Z_n - 1) \sum_{k=1}^n (Z_k-1) \right) \]
	and so~\eqref{eq:variance-Z} gives $v_n = 2n + v_{n-1}$. In particular, $v_n$ grows quadratically. On the other hand, we have
	\[ \Var(Z_1 + \cdots + Z_{n+1}) = 2(n+1) + 2 v_n + \Var(Z_1 + \cdots + Z_n), \]
	and since $v_n$ grows quadratically in $n$, this implies that $\Var(Z_1 + \cdots + Z_n)$ grows like $n^3$, which proves the result.
\end{proof}

\end{document}